\documentclass[11pt]{amsart}
\usepackage{amsfonts,amssymb,graphicx, enumerate,ytableau}
\usepackage{tikz}

\input prepictex
\input pictex
\input postpictex

\numberwithin{equation}{section}

\newtheorem{theorem}{Theorem}[section]
\newtheorem{lemma}[theorem]{Lemma}

\newtheorem{corollary}[theorem]{Corollary}

\theoremstyle{definition}

\theoremstyle{remark}

\newcommand{\CC}{\mathbf{C}}
\newcommand{\HH}{\mathbb{H}}
\newcommand{\RR}{\mathbf{R}}

\newcommand{\ZZ}{\mathbf{Z}}

\newcommand{\OO}{\mathcal{O}}

\newcommand{\la}{\langle}
\newcommand{\ra}{\rangle}

\newcommand{\ttt}{\mathfrak{t}}

\begin{document}

\title{Unitary representations of cyclotomic rational Cherednik algebras}

\author{Stephen Griffeth}
\address{Stephen Griffeth \\
Instituto de Matem\'atica y F\'isica \\
Universidad de Talca \\
Talca, Chile}
\email{sgriffeth@inst-mat.utalca.cl}

\begin{abstract}
We classify the irreducible unitary modules in category $\OO_c$ for the rational Cherednik algebras of type $G(r,1,n)$ and give explicit combinatorial formulas for their graded characters. More  precisely, we produce a combinatorial algorithm determining, for each $r$-partition $\lambda^\bullet$ of $n$,  the closed semi-linear set of parameters $c$ for which the contravariant form on the irreducible representation $L_c(\lambda^\bullet)$ is positive definite. We use this algorithm to give a closed form answer for the Cherednik algebra of the symmetric group (recovering a result of Etingof-Stoica and the author) and the Weyl groups of classical type.
\end{abstract}

\thanks{I am especially indebted to Arun Ram and Ivan Cherednik for teaching me, in person and through their writing, the techniques employed in this paper.  I thank Pavel Etingof, Emanuel Stoica, and Takeshi Suzuki for useful suggestions which have improved the paper. I acknowledge the financial support of Fondecyt Proyecto Regular 1151275 and the Mathamsud grant Rephomol.}

\maketitle

\section{Introduction}

\subsection{}The goals of this paper are: first, to obtain the classification of irreducible unitary representations in category $\OO_c$ for the rational Cherednik algebras of type $W=G(r,1,n)$, and second, to give an explicit basis for each of them in terms of Specht-module valued versions of non-symmetric Jack polynomials. The strategy is that of the appendix to \cite{EtSt}.  The tactics are somewhat different, requiring the development of tools that were previously unavailable for $r>1$, and the answer is quite a bit more complicated. The main steps are as follows: first, in Theorem \ref{trig pres} we give a presentation of the cyclotomic rational Cherednik algebra compatible with a certain commutative subalgebra discovered by Dunkl and Opdam \cite{DuOp}; second, in the proof of Theorem \ref{diag thm}, we use this presentation to extend the results of Cherednik \cite{Che1} and Suzuki \cite{Suz}  classifying the diagonalizable representations for the type A Cherednik algebra to the cyclotomic case; third, in the proof of Theorem \ref{unitary thm} we use Theorem \ref{diag thm} together with arguments analogous to those in the appendix of \cite{EtSt} to complete the classification of the unitary irreducible objects in category $\OO_c$. 

The irreducible objects in category $\OO_c$ are indexed by the irreducible complex representations of the complex reflection group $W$. Thus for $W=S_n$ the symmetric group, they are indexed by integer partitions of $n$, and for $W=G(2,1,n)=W(B_n)$ the Weyl group of type $B_n$ they are indexed by pairs $(\lambda^0,\lambda^1)$ of integer partitions with $n$ total boxes. More generally, for the complex reflection group $G(r,1,n)$ they are indexed by $r$-tuples $(\lambda^0,\dots,\lambda^{r-1})$ of integer partitions with $n$ total boxes. In the case $W=S_n$, the parameter $c$ is a single real number; for the Weyl group of type $B$ it is a pair $(c,d)$ consisting of two real numbers; and for the complex reflection group $G(r,1,n)$ it is an $r$-tuple of real numbers. Our main theorem provides an algorithm determing, for each $r$-partition $\lambda$ of $n$, the closed semi-linear set of parameters $U(\lambda^\bullet) \subseteq \RR^r$ of points for which the contravariant form on the irreducible object $L_c(\lambda^\bullet)$ is positive definite. Our algorithm may be made into a closed form answer for $r=1$ and $r=2$ (and, presumably, also for any fixed $r$, given enough patience), and we record this closed form answer in Corollary \ref{B corollary 1} and Corollary \ref{B corollary 2}. 

We recover the main result of \cite{EtSt} for $r=1$. Already for $r=2$ the answer is much more intricate: even to state it in completely explicit fashion requires several pages (see Section \ref{classical type}). To give the reader a rough idea of the form of the answer and a visual demonstration of the varying levels of complexity, here we present miniature drawings, all to the same scale, of the unitary sets for three particular cases: first, for the group $W=S_{30}$ and the partition $\lambda=(6,6,6,4,4,4)$. The set of all possible $c$ is one-dimensional, and the unitary set consists of a closed interval, and five isolated points.

\begin{center}
\begin{tikzpicture}[scale=8]
\tikzstyle{axes}=[]
\begin{scope}[style=axes]
\draw[<->] (-1/4,0) -- (1/4,0) node[right] {$c$} coordinate(c1 axis);
\end{scope}
\begin{scope}[very thick,blue, auto=left]
\draw (-1/11,0) -- (1/11,0);
\draw[fill] (1/10,0) circle (.05 pt);
\draw[fill] (1/9,0) circle (.05 pt);
\draw[fill] (-1/10,0) circle (.05 pt);
\draw[fill] (-1/9,0) circle (.05 pt);
\draw[fill] (-1/8,0) circle (.05 pt);
\end{scope}
\end{tikzpicture}
\end{center} The interior of the closed interval is the set on which the standard module itself is unitary.

Second, for the group $W=G(2,1,30)$ and the bipartition $\lambda=((6,6,6,4,4,4),\emptyset)$, the unitary set is the subset of the $(c,d)$ plane shown here:

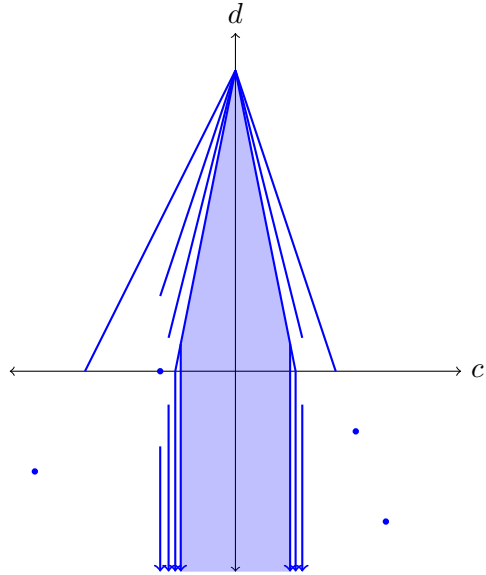
\begin{center}
\begin{tikzpicture}[scale=8]
\tikzstyle{axes}=[]
\tikzstyle{wall}=[thick]
\tikzstyle{relevant wall}=[very thick]
\tikzstyle{dot}=[fill]
\begin{scope}[style=axes]
\draw[<->] (-3/8,0) -- (3/8,0) node[right] {$c$} coordinate(c1 axis);
\draw[<->] (0,-1/3) -- (0,9/16) node[above] {$d$} coordinate(c2 axis);
\end{scope}

\begin{scope}[thick,blue,auto=left]
\draw (0, 1/2) -- (1/10,0);
\draw (0, 1/2) -- (-1/10,0);
\draw (0,1/2) -- (1/9,1/18);
\draw (0,1/2) -- (1/6,0);
\draw (0,1/2) -- (-1/9,1/18);
\draw (0,1/2) -- (-1/8,1/8);
\draw (0,1/2) -- (-1/4,0);
\draw[fill] (1/5,-1/10) circle (.1pt);
\draw[fill] (1/4,-1/4) circle (.1pt);
\draw[fill] (-1/3,-1/6) circle (.1pt);
\draw[fill] (-1/8,0) circle (.1pt);
\draw[->] (1/11,1/22 ) -- (1/11,-1/3);
\draw[->] (-1/11,1/22 ) -- (-1/11,-1/3);
\draw[->] (-1/10,0) -- (-1/10,-1/3);
\draw[->] (-1/9,-1/18 ) -- (-1/9,-1/3);
\draw[->] (-1/8,-1/8 ) -- (-1/8,-1/3);
\draw[->] (1/10,0) -- (1/10,-1/3);
\draw[->] (1/9,-1/18) -- (1/9,-1/3);
\fill[blue, nearly transparent] (0, 1/2) -- (1/11,1/22) -- (1/11, -1/3) -- (-1/11,-1/3) -- (-1/11,1/22) -- cycle;
\end{scope}

\end{tikzpicture}
\end{center} The shaded region is the set on which the standard module itself is unitary; on the darker lines and points, the standard module is not irreducible, and its proper quotient $L_c(\lambda)$ is unitary.

Finally, for the group $W=G(2,1,34)$ and the bipartition $((3,3,3,2,2),(5,5,5,5,3,3))$, the unitary set looks like this (as above, the shaded parallelogram is the set of parameters for which the standard module itself is unitary):

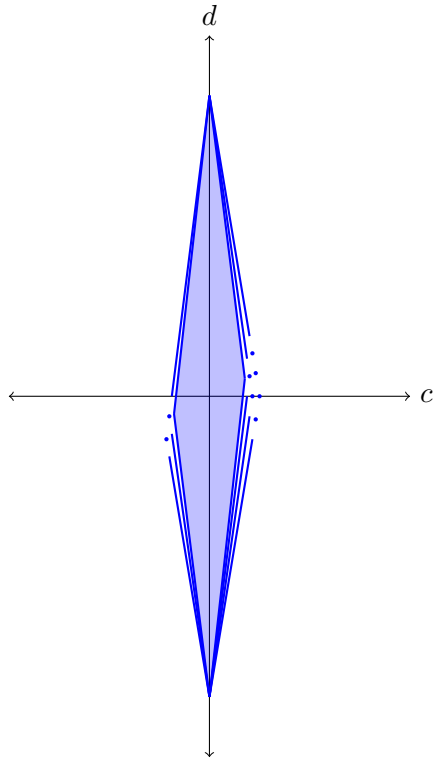
\begin{center}
\begin{tikzpicture}[scale=8]
\tikzstyle{axes}=[]
\tikzstyle{wall}=[thick]
\tikzstyle{relevant wall}=[very thick]
\tikzstyle{dot}=[fill]
\begin{scope}[style=axes]
\draw[<->] (-1/3,0) -- (1/3,0) node[right] {$c$} coordinate(c1 axis);
\draw[<->] (0,-3/5) -- (0,3/5) node[above] {$d$} coordinate(c2 axis);
\end{scope}

\begin{scope}[thick,blue,auto=left]
\fill[blue, nearly transparent] (0, 1/2) -- (1/17,1/34) -- (0,-1/2) -- (-1/17,-1/34) -- cycle;
\draw (0, 1/2) -- (1/17,1/34);
\draw (0, -1/2) -- (1/17,1/34);
\draw (0, 1/2) -- (-1/17,-1/34);
\draw (0, -1/2) -- (-1/17,-1/34);
\draw (0,1/2) -- (1/16,1/16);
\draw (0,1/2) -- (1/15,1/10);
\draw (0,-1/2) -- (1/16,0);
\draw (0,-1/2) -- (1/15,-1/30);
\draw (0,-1/2) -- (1/14,-1/14);
\draw (0,1/2) -- (-1/16,0);
\draw (0,-1/2) -- (-1/16,-1/16);
\draw (0,-1/2) -- (-1/15,-1/10);
\draw[fill] (1/15,1/30) circle (.05 pt);
\draw[fill] (1/14,0) circle (.05 pt);
\draw[fill] (1/14,1/14) circle (.05 pt);
\draw[fill] (1/13,1/26) circle (.05 pt);
\draw[fill] (1/13,-1/26) circle (.05 pt);
\draw[fill] (1/12,0) circle (.05 pt);
\draw[fill] (-1/15,-1/30) circle (.05 pt);
\draw[fill] (-1/14,-1/14) circle (.05 pt);
\end{scope}
\end{tikzpicture}
\end{center}

The Cherednik algebra $H_c$ contains a commutative subalgebra $\ttt$, the \emph{Dunkl-Opdam subalgebra}, that acts by locally finite, normal, and hence diagonalizable, operators on any unitary representation in category $\OO_c$.  Therefore a first step towards the classification of irreducible unitary modules in $\OO_c$ is the classification of the irreducible modules in $\OO_c$ on which $\ttt$ acts by diagonalizable operators (A. Ram \cite{Ram} uses the word \emph{calibrated} for the analog of this type of module for the affine Hecke algebra).  Once we have achieved this classification (Theorem~\ref{diag thm}), the classification of unitary modules (Theorem \ref{unitary thm}) is obtained by Cherednik's technique of intertwining operators, combined with a detailed study of the combinatorics of certain tableaux.  For the groups $W=G(2,1,n)$ we can be completely explicit, but for larger $r$ it seems that a direct description of the unitary irreducibles in $\OO_c$ is unavoidably complicated. In particular, the set of $c$ for which $L_c(\lambda)$ is unitary may have components of all dimensions between $0$ and $r$. 

In the appendix to \cite{EtSt} we had the advantage that the first part (classification of diagonalizable modules) had been previously carried out by Cherednik and Suzuki for the trigonometric DAHA, and the corresponding classification for the rational DAHA relied on an embedding of the latter into the former. Here we first obtain the classification of diagonalizable modules by working directly with the rational DAHA, making use of a presentation adapted to the technique of intertwining operators (since the first version of this paper appeared on the arxiv, this presentation was rediscovered by Webster \cite{Web}, and Braverman-Etingof-Finkelberg \cite{BEF} have constructed a cyclotomic version of the full DAHA).  Once this is done there are necessary and sufficient numerical criteria that the eigenvalues of a diagonalizable module must satisfy in order that it be unitary, and the remainder of the paper is devoted to constructing a combinatorial machine for handling these numerics.

We now introduce the notation we will need to state our main results, referring to Section \ref{background} for precise definitions. We identify integer partitions with their Young diagrams. Given a box $b$ of a partition, we will write $\mathrm{ct}(b)$ for its content, equal to $i-j$ if the box is in column $i$ and row $j$. Given an $r$-partition $\lambda^\bullet=(\lambda^0,\dots,\lambda^{r-1})$ and a box $b \in \lambda^i$ we will write $\beta(b)=i$. The Cherednik algebra $H_c$ depends on a parameter $c=(c_0,d_0,d_1,\dots,d_{r-1}) \in \RR^{r+1}$, where $d_0+d_1+\cdots+d_{r-1}=0$. It contains a certain commutative subalgebra $\ttt$ (see \ref{t def} for the precise definition) which acts by normal operators, and hence diagonalizably, on any unitary representation. Given a parameter $c$ and an $r$-partition $\lambda^\bullet$, we write $L_c(\lambda^\bullet)$ for the corresponding representation of $H_c$. We follow the convention that the superscript $i$ in $\lambda^i$ is always to be taken modulo $r$, as is the subscript $i$ in $d_i$ (so that $\lambda^i$ and $d_i$ are defined for all integers $i \in \ZZ$).

Our first main result is the classification and description of the modules $L_c(\lambda^\bullet)$ that are $\ttt$-diagonalizable.  For each box $b \in \lambda^\bullet$, define statistics $c(b)$, $k_c(b)$ and
$l_c(b)$ as follows: first, $c(b)$ is the \emph{charged content} of $b$, defined by
\begin{equation}
c(b)=d_{\beta(b)}+r \mathrm{ct}(b)c_0.
\end{equation} Second, $k_c(b)$ is the smallest positive integer $k$ such
that there is a box $b' \in \lambda^{\beta(b)-k}$ with
$$k=c(b)-c(b'),$$
(and $k_c(b)=\infty$ if no such equation holds) and finally, $l_c(b)$ is the smallest positive integer $l$ such that there is an outside addable box $b'$ for $\lambda^{\beta(b)-l}$ with
$$l=c(b)-c(b')$$  (and $l_c(b)=\infty$ if no such equation holds).  Recall that a box $b$ is \emph{addable} to a partition $\lambda$ if $b \notin \lambda$ and adding $b$ to $\lambda$ produces the diagram of a partition.  An \emph{outside} addable box is an addable box $b$ such that $\mathrm{ct}(b) \neq \mathrm{ct}(b')$ for all $b' \in \lambda$.  

Given boxes $b,b' \in \lambda^i$ we write $b \leq b'$ if $b$ is (weakly) up and to the left of $b'$ (thus in particular $\beta(b)=\beta(b')$ if $b \leq b'$). We define $\Gamma$ to be the set of pairs $(P,Q)$ where $P$ and $Q$ are fillings of the boxes of $\lambda^\bullet$ by non-negative integers, $P$ is a bijection from the boxes of $\lambda^\bullet$ to the set $\{1,2, \dots,n \}$, $Q$ is weakly increasing $Q(b) \leq Q(b')$ if $b \leq b'$, and $P(b) > P(b')$ if $b \leq b'$ and $Q(b)=Q(b')$. 

We define a subset $\Gamma_c \subseteq \Gamma$ as follows: a pair $(P,Q) \in \Gamma$ is in $\Gamma_c$ if and only if the following conditions hold:
\begin{enumerate}
\item[(a)] whenever $b \in \lambda^\bullet$
and $k \in \ZZ_{>0}$ with $k=d_{\beta(b)}-d_{\beta(b)-k}+r \mathrm{ct}(b) c_0$ we have $Q(b)<k$, and 
\item[(b)] whenever $b_1,b_2 \in
\lambda^\bullet$ and $k \in \ZZ_{>0}$ with $\beta(b_1)-\beta(b_2)=k$ mod $r$ and
$k=d_{\beta(b_1)}-d_{\beta(b_2)}+r (\mathrm{ct}(b_1)-\mathrm{ct}(b_2)
\pm 1)c_0$ we have $Q(b_1) \leq Q(b_2)+k$, with equality implying
$P(b_1)>P(b_2)$. \end{enumerate}

\begin{theorem} \label{diag thm}
The module $L_c(\lambda^\bullet)$ is diagonalizable if and only if either
\begin{enumerate}
\item[(a)] $c_0=0$ or 
\item[(b)] $c_0 \neq 0$ and for every removable box $b \in \lambda^\bullet$, either $k_c(b)=\infty$ or the inequality $l_c(b)<k_c(b)$ holds.
\end{enumerate}  In case (b), as basis of $L_c(\lambda^\bullet)$ is given by the set $\{f_{(P,Q)} \ | \ (P,Q) \in \Gamma_c \}$ of non-symmetric Specht-valued Jack polynomials.
\end{theorem}

Each irreducible representation $L_c(\lambda^\bullet)$ is equipped with a non-degenerate Hermitian contravariant form. We call $L_c(\lambda)$ \emph{unitary} if this form is positive definite. Our second main result is the classification of the modules $L_c(\lambda^\bullet)$ that are unitary.  We state the result for $c_0 \geq 0$: $L_c(\lambda^\bullet)$ is unitary if and only if $L_{c'}((\lambda^\bullet)^t)$ is, where $c'$ is the same as $c$ but with $c_0$ replaced by $-c_0$, and the transpose of an $r$-partition is the transpose of each of its component $\lambda^j$'s. 

For $0 \leq i , j \leq r-1$ we write $m_{ij}$ for the integer with $1 \leq m_{ij} \leq r$ and $m_{ij}=i-j \ \mathrm{mod} \ r$ (thus for $i=j$ we have $m_{ij}=r$). Let $(b,j)$ be a pair consisting of a box $b \in \lambda^i$ and an integer $0 \leq j \leq r-1$. A \emph{blocking sequence} $B$ for $(b,j)$ is a sequence $B=(b_1,b_2,\dots,b_{2k+1},\ell)$ of boxes and an integer $0 \leq \ell \leq r-1$ such that 
\begin{enumerate}
\item[(a)] $b \leq b_1$, and for each $1 \leq p \leq k$, $b_{2p} \leq b_{2p+1}$, and 
\item[(b)] $m_{\beta(b_{2k+1}),\ell}+\sum_{p=1}^k m_{\beta(b_{2p-1}),\beta(b_{2p})} \leq m_{ij}$.
\end{enumerate} 

Condition (b) may be rephrased and visualized as follows: linearly order the set of integers modulo $r$ so that $i-1>i-2> \cdots > i$. Then the sequence $(\beta(b_{2}),\beta(b_4),\dots,\ell)$ must be a decreasing subsequence of the interval $i-1>i-2> \cdots > j+1 > j$. 

Given a blocking sequence $B=(b_1,b_2,\dots,b_{2k+1},\ell)$ for $(b,j)$, we define the set $L_B$ of parameters by $c \in L_B$ if and only if
$$ d_{\beta(b_{2p-1})}-d_{\beta(b_{2p})}+r(\mathrm{ct}(b_{2p-1})-\mathrm{ct}(b_{2p})\pm 1)c_0=m_{\beta(b_{2p-1}),\beta(b_{2p})} \quad \hbox{for $1 \leq p \leq k$}$$ and
$$d_{\beta(b_{2k+1})}-d_\ell+r \mathrm{ct}(b_{2k+1}) c_0=m_{\beta(b_{2k+1}),\ell}.$$ 

Given an ordered pair of boxes $(b,b')$ with $b \in \lambda^i$ and $b' \in \lambda^j$, a \emph{blocking sequence} for $(b,b')$ is either a blocking sequence for $(b,j)$, or a sequence $B=(b_1,\dots,b_{2q})$ of boxes of $\lambda^\bullet$ such that 
\begin{enumerate}
\item[(a)] $b \leq b_1$ and $b_{2q} \leq b'$,
\item[(b)] for each $1 \leq k \leq q-1$ we have $b_{2k} \leq b_{2k+1}$, and
\item[(c)] we have
$$\sum_{k=1}^ {q} m_{\beta(b_{2k-1}),\beta(b_{2k})} \leq m_{ij}.$$
\end{enumerate} Just as above, condition (c) is equivalent to the requirement that the sequence $(\beta(b_2),\beta(b_4),\dots,\beta(b_{2q}))$ be an decreasing subsequence of the interval $i-1>i-2>\cdots>j+1> j$.

Given a blocking sequence $B$ for $(b,b')$ of this second type, we define  the set $L_B$ of parameters by $c \in L_B$ if and only if
$$ d_{\beta(b_{2p-1})}-d_{\beta(b_{2p})}+r(\mathrm{ct}(b_{2p-1})-\mathrm{ct}(b_{2p})\pm 1)c_0=m_{\beta(b_{2p-1}),\beta(b_{2p})} \quad \hbox{for $1 \leq p \leq q$}.$$

\begin{theorem} \label{unitary thm}
For $c_0=0$, the module $L_c(\lambda^\bullet)$ is unitary if and only if for all $0 \leq i \leq r-1$ such that $\lambda^i \neq \emptyset$ we have $d_i-d_j \leq m_{ij}$ for all $0 \leq j \leq r-1$. Suppose $c_0 > 0$.  The module $L_c(\lambda^\bullet)$ is unitary if and only if the following conditions are satisfied:
\begin{enumerate}
\item[(a)] it is diagonalizable (see Theorem~\ref{diag thm}),
\item[(b)] for every pair $b,b' \in \lambda^\bullet$ of boxes such that $$d_{\beta(b)}-d_{\beta(b')}+r(\mathrm{ct}(b)-\mathrm{ct}(b')+1)c_0>m_{\beta(b),\beta(b')},$$ either $b \leq b'$ or else there is a blocking sequence $B$ with $c \in L_B$, and
\item[(c)] for every box $b \in \lambda^\bullet$ and $0 \leq j \leq r-1$ such that $$d_{\beta(b)}-d_j+r \mathrm{ct}(b) c_0>m_{\beta(b),j}$$ there is a blocking sequence $B$ with $c \in L_B$.
\end{enumerate}
\end{theorem}

The theorem exhibits the set of parameters $c$ for which the module $L_c(\lambda^\bullet)$ is unitary as a semi-linear subset of the parameter space. At first sight it might seem to require an analysis of an unmanageable number of inequalities and equalities, but it is a practical computational tool: in Section \ref{classical type} we will deduce the classification theorem of \cite{EtSt} from it, and use it to explicitly compute the unitary spectrum of each $L_c(\lambda^\bullet)$ for Weyl groups of classical type.

With Theorems \ref{diag thm} and \ref{unitary thm} in hand, the natural next step is the calculation of the Kazhdan-Lusztig polynomials
$$P_{\lambda,\mu}(q)=\sum \mathrm{dim}(\mathrm{Ext}^i(\Delta_c(\lambda),L_c(\mu))) q^i$$ for all unitary representations $L_c(\lambda^\bullet)$. Using Dirac cohomology and the Hodge decomposition theorem from \cite{HuWo} (c.f. \cite{Ciu}) coupled with our character formula in Theorem \ref{diag thm} should produce a combinatorial algorithm (much faster than the usual linear algebra using canonical bases) based on jeu de taquin for computing these KL polynomials. We expect that in fact this algorithm simplifies greatly; in the type A case this simplification should prove the formula for the Betti numbers of the $k$-equals ideal conjectured in \cite{BGS}. 

T. Suzuki remarks that in the type A case, the unitary modules correspond to integrable modules and the diagonalizable modules correspond to the \emph{admissible} modules of Kac and Wakimoto via the Arakawa-Suzuki functor.  We do not know the analogs of this coincidence for the groups $G(r,p,n)$, though it should be interesting to compare our results with those of Varagnolo and Vasserot, \cite{VaVa}. Finally, we remark that the version of Clifford theory that appears in the last section of \cite{Gri2} allows one to deduce analogous results for the groups $G(r,p,n)$ when $n>2$. 

\section{Definitions and background} \label{background}

\subsection{} Let $V$ be a complex vector space of dimension $n$ and let $G \subseteq \mathrm{GL}(V)$ be a finite group of linear transformations of $V$.  The set of \emph{reflections} (sometimes called pseudo-reflections or complex reflections) in $G$ is
$$R=\{r \in G \ | \ \mathrm{dim}(\mathrm{fix}(r))=n-1 \}.$$  The group $G$ is a \emph{reflection group} if it is generated by $R$.  

\subsection{}  For each reflection $r \in R$ let $c_r$ be a formal variable such that $c_{g r g^{-1}}=c_r$ for all $r \in R$ and $g \in G$, and choose $\alpha_r \in V^*$ such that the zero set of $\alpha_r$ is the fix space of $r$.  Let $A=\CC[c_r]_{r \in R}$ be the ring of polynomials generated by these variables (thus $A$ is a polynomial ring in a set of variables corresponding to the conjugacy classes of reflections).  Write $A[V]=A \otimes_\CC \CC[V]$ for the ring of polynomial functions on $V$ with coefficients in $A$.  For each $y \in V$ define a \emph{Dunkl operator} on $A[V]$ by 
\begin{equation} \label{Dunkl op} 
y(f)=\partial_y(f)-\sum_{r \in R} c_r \la \alpha_r,y \ra \frac{f-r(f)}{\alpha_r} \quad \hbox{for $f \in A[V]$,}
\end{equation} where $\partial_y$ is the partial derivative of $f$ in the direction $y$ and $\la \cdot,\cdot \ra$ is the natural pairing between $V^*$ and $V$.  Each $f \in A[V]$ defines a multiplication operator $h \mapsto fh$ on $A[V]$, and the \emph{rational Cherednik algebra} $H=H(G,V)$ determined by these data is the $A$-subalgebra of $\mathrm{End}_A(A[V])$ generated by $G$, $A[V]$, and the Dunkl operators $y \in V$.  
 
A routine computation shows that these operators satisfy the relations
\begin{equation} \label{rel 1}
g y g^{-1}=g(y) \quad g x g^{-1}=g(x) \quad \hbox{for $g \in G$, $x \in V^*$, and $y \in V$,}
\end{equation} and by induction on the degree of $f$
\begin{equation} \label{rel 2}
yf-fy=\partial_y(f)-\sum_{r \in R} c_r \la \alpha_r,y \ra \frac{f-r(f)}{\alpha_r} r \quad \hbox{for $y \in V$ and $f \in A[V]$.}
\end{equation} Using these relations one then checks
$$[[y_1,y_2],f]=[y_1,[y_2,f]]+[[y_1,f],y_2]=0$$ which implies 
\begin{equation}
y_1 y_2=y_2 y_1 \quad \hbox{for all $y_1,y_2 \in V$.}
\end{equation}  In fact $H$ is generated by $A[V]$, $W$, and $A[V^*]$ subject to \eqref{rel 1} and the special case of \eqref{rel 2} in which $f \in V^*$ is linear, in which case it may be written as
\begin{equation} \label{xy rel}
yx-xy=\la x,y \ra-\sum_{r \in R} c_r \la \alpha_r,y \ra \la x,\alpha_r^\vee \ra r
\end{equation} where $x-r(x)=\la x,\alpha_r^\vee \ra \alpha_r$ determines $\alpha_r^\vee \in V$. 

Multiplication induces an isomorphism
\begin{equation} \label{triangular}
A[V] \otimes AW \otimes A[V^*] \rightarrow H
\end{equation} called the \emph{triangular decomposition} of $H$.

\subsection{}  Let $\{S^\lambda \ | \ \lambda \in \Lambda \}$ be the set of irreducible representations of $\CC W$, where $\Lambda$ is an index set. To avoid a profusion of subscripts, we abuse notation and write also $S^\lambda$ for its extension to $A W$. Write $A[V^*] \rtimes W$ for the subalgebra of $H$ generated by the Dunkl operators $y \in V$ and the group $W$.  The \emph{standard module} corresponding to $\lambda \in \Lambda$ is
\begin{equation}
\Delta(\lambda)=\mathrm{Ind}^H_{A[V^*] \rtimes W} (S^\lambda)
\end{equation}  where the $A[V^*] \rtimes W$-module structure on $S^\lambda$ is determined by $y S^\lambda=0$ for all $y \in V$.  Thanks to the triangular decomposition \eqref{triangular} there is an isomorphism of $A[V] \rtimes W$-modules
\begin{equation}
A[V] \otimes_A S^\lambda \rightarrow \Delta(\lambda) 
\end{equation} and via this isomorphism the Dunkl operators act according to the formula
\begin{equation} \label{spin Dunkls}
y(f \otimes v)=\partial_y(f) \otimes v-\sum_{r \in R} c_r \la \alpha_r,y \ra \frac{f-s(f)}{\alpha_r} \otimes s(v) \quad \hbox{for $y \in V$, $f \in A[V]$, and $v \in S^\lambda$.}
\end{equation}

\subsection{} There is a reparametrization that simplifies the expression of many numbers arising naturally in the study of the Cherednik algebra (in particular, the eigenvalues of the monodromy of the connections corresponding to the standard modules).  Let $\mathcal{A}$ be the set of hyperplanes $H$ in $V$ of the form $H=\mathrm{fix}(r)$ for some $r \in R$.  For each $H \in \mathcal{A}$ choose $\alpha_H \in V^*$ with $H$ equal to the zero set of $\alpha_H$.  The subgroup $W_H=\{w \in W \ | \ w(v)=v \ \hbox{if $v \in H$} \}$ is a cyclic subgroup, and we write $W_H^\vee$ for its character group.  Let $n_H=|W_H|$ be the size of $W_H$ and for $\chi \in W_H^\vee$ let 
$$e_{H,\chi}=\frac{1}{n_H}\sum_{w \in W_H} \chi(w^{-1}) w \in \CC W_H$$ be the corresponding primitive idempotent.  Define $c_{H,\chi}$ by
\begin{equation} \label{cH def}
c_{H,\chi} n_H=\sum_{r \in W_H-\{1\}} c_r(1-\chi(r)).
\end{equation} In particular $c_{H,\mathrm{triv}}=0$, and $A$ may be viewed as a polynomial ring in the variables $c_{H,\chi}$ (modulo the relations $c_{H,\chi}=c_{wH,w\chi w^{-1}}$ for $w \in W$).  Using the relation
$$w=\sum_{\chi \in W_H^\vee} \chi(w) e_{H,\chi} \quad \hbox{for $w \in W_H$,}$$ the formula for Dunkl operators becomes
$$y(f)=\partial_y(f)-\sum_{H \in \mathcal{A}} \frac{\la \alpha_H,y \ra}{\alpha_H} \sum_{\chi \in W_H^\vee-\{1 \}}  c_{H,\chi} n_H e_{H,\chi},$$ which is, up to a sign, the formula in \cite{DuOp}, equation (5).  

For each $H \in \mathcal{A}$ fix an eigenvector $\alpha_H^\vee \notin H$ for $W_H$.  In terms of the parameters $c_{H,\chi}$ the relation \eqref{xy rel} is
\begin{equation}
yx-xy=\la x,y \ra - \sum_{H \in \mathcal{A}} \frac{ \la \alpha_H,y \ra \la x,\alpha_H^\vee \ra}{\la \alpha_H,\alpha_H^\vee \ra}  \sum_{\chi \in W_H^\vee}  (c_{H,\chi \otimes \mathrm{det}^{-1}}-c_{H,\chi}) n_H e_{H,\chi}
\end{equation}

\subsection{} We extend complex conjugation to $A=\CC[c_{H,\chi}]$ be declaring $\overline{c_{H,\chi}}=c_{H,\chi}$. Fix  a positive definite Hermitian inner product $\la \cdot,\cdot \ra$ on $S^\lambda$ and mutually inverse $W$-equivariant conjugate linear isomorphisms $V \rightarrow V^*$ and $V^* \rightarrow V$, written $y \mapsto y^*$ and $x \mapsto x^*$.  Then $\la \cdot, \cdot \ra$ has a unique extension, also denoted $\la \cdot,\cdot \ra$, to $\Delta(\lambda)$ determined by the following rules: 
\begin{enumerate}
\item[(a)] $\la \cdot,\cdot \ra$ is bi-additive, $A$-linear in the second variable, and $A$-conjugate linear in the first variable with respect to the extension of complex conjugation to $A$ that fixes the variables $c_{H,\chi}$,
\item[(b)] $\la x f, g \ra=\la f, x^*g \ra$ for all $x \in V^*$ and $f,g \in \Delta(\lambda)$.
\end{enumerate}  

\subsection{}  We will consider two types of extensions of scalars: first, writing $F=\mathrm{Frac}(A)$ for the fraction field of $A$, we write $H_F=F \otimes_A H$ for the \emph{generic} Cherednik algebra, and similarly $\Delta_F(\lambda)=F \otimes_R \Delta(\lambda)$ and $\la \cdot,\cdot \ra_F$ for the $F$-conjugate linear extension of $\la \cdot,\cdot \ra$ to  $\Delta_F(\lambda)$.  Second, given a specialization $A \rightarrow \CC$ of the variables $c_r$ (or $c_{H,\chi}$) to complex numbers, we will write $H_c=\CC \otimes_A H$ and $\Delta_c(\lambda)=\CC \otimes_A \Delta(\lambda)$ for the corresponding specializations.  We think of the symbol $c$ as standing for this specialization, or equivalently, for a set of complex numbers indexed by conjugacy classes of reflections (or conjugacy classes of characters of rank one parabolic subgroups).  In case the specialization is such that the variables $c_{H,\chi}$ are all real, the contravariant form also specializes and we write $\la \cdot,\cdot \ra_c$ for its specialization.

\subsection{}  Suppose that we have specialized the variables to complex numbers $c$.  \emph{Category $\OO_c$} is the subcategory of the category $H_c \mathrm{-mod}$ of finitely generated $H_c$-modules on which each Dunkl operator $y \in V$ acts locally nilpotently.  Thanks to \eqref{spin Dunkls} each standard module $\Delta_c(\lambda)$ is in $\OO_c$.  In fact, the quotient $L_c(\lambda)$ of the standard module $\Delta_c(\lambda)$ by its radical is simple and this gives a complete list of inequivalent irreducible objects in $\OO_c$.  Furthermore, the radical of $\Delta_c(\lambda)$ coincides with the radical of the form $\la \cdot,\cdot \ra_c$.  Therefore the contravariant form descends to a non-degenerate form on $L_c(\lambda)$, and Cherednik has posed the problem of deciding when this form is positive definite. The study of this problem began with the paper \cite{EtSt}. 

\subsection{}  From now on, we will assume that $G=W$ is a monomial group.  The precise definitions follow.  Fix positive integers $r$ and $n$.  Let $W=G(r,1,n)$ be the group of $n$ by $n$ matrices such that the entries are either $0$ or a power of $e^{2 \pi \sqrt{-1}/r}$, and there is exactly one non-zero entry in each row and each column.  Then $W$ is a group of matrices acting on $V=\CC^n$, and we write $y_1,\dots,y_n$ for the standard basis of $V$.

There are two $W$-orbits of reflecting hyperplanes: writing $x_1,\dots,x_n$ for the standard basis of $V^*$, there those of the form $x_i=0$ (for which $n_H=r$) and those of the form $x_i=\zeta^l x_j$ (for which $n_H=2$).  Write $s_{ij}$ for the permutation matrix interchanging $i$ and $j$ and fixing all other coordinates, and $\zeta_i$ for the matrix that multiplies the $i$th coordinate by $\zeta=e^{2 \pi \sqrt{-1}/r}$ and fixes all other coordinates.  We will write $e_{ij}=\frac{1}{r} \sum_{l=0}^{r-1} \zeta^{-lj} \zeta_i^l$, and leave the other idempotents unnamed.  

\subsection{}  In the case $W=G(r,1,n)$ the relations for the rational Cherednik algebra may be written in the following extremely explicit form.  Let $c_0$ and $d_1,\dots,d_{r-1} \in \CC$ be variables and let $A=\CC[c_0,d_1,\dots,d_{r-1}]$.  The Cherednik algebra $H$ for $W$ is generated by the algebras $A[y_1,\dots,y_n]$, $R[x_1,\dots,x_n]$, and $A W$, subject to the relations $w f w^{-1}=w(f)$ for $f \in A[y_1,\dots,y_n]$ or $f \in A[x_1,\dots,x_n]$ and $w \in W$,
 \begin{equation}
 y_i x_i=x_i y_i-c_0 \sum_{\substack{ 1 \leq j \neq i \leq n \\ 0 \leq l \leq r-1}} \zeta_i^l s_{ij} \zeta_i^{-l}-\sum_{l=0}^{r-1} (d_l-d_{l-1}) e_{il}
 \end{equation} for $1 \leq i \leq n$, and
\begin{equation}
y_i x_j=x_j y_i+ c_0 \sum_{l=0}^{r-1} \zeta^{-l} \zeta_i^l s_{ij} \zeta_i ^{-l}
\end{equation} for $1 \leq i \neq j \leq n$.  To define $d_l$ for all $l \in \ZZ$ we specify $d_0$ by the relation $d_0+d_1+\cdots+d_{r-1}=0$ and impose $d_i=d_j$ if $i=j \ \mathrm{mod} \ r$.  In terms of the parameters $c_r$ attached to conjugacy classes of reflections, $c_0$ is the parameter for the conjugacy class of transpositions, and
$$d_j=\sum_{1 \leq m \leq r-1} \zeta^{jm} c_m,$$ where $c_m$ is the conjugacy class of reflections containing $\zeta_1^m$. Comparing with \eqref{cH def}, this implies that $c_0$ and $d_0,d_1,\dots,d_{r-1}$ are all real if and only if the $c_{H,\chi}$'s are all real.  Whenever we make use of the specialized contravariant form, we will assume the parameters $c_0$ and $d_l$ are all real.
 
\subsection{}  \label{partitions subsection} A \emph{partition} $
\lambda=\lambda_1 \geq \lambda_2 \geq \cdots$ is a weakly decreasing sequence of integers such that $\lambda_n=0$ for $n$ large enough.  Given a positive integer $r$, an $r$-partition is a sequence $\lambda^\bullet=(\lambda^0,\lambda^1,\dots,\lambda^{r-1})$ of $r$ partitions.  The \emph{size} of an $r$-partition $\lambda^\bullet$ is the sum $|\lambda^\bullet|=\sum_{i,j} \lambda^i_j$, and an \emph{$r$-partition of $n$} is an $r$-partition $\lambda^\bullet$ with $|\lambda^\bullet|=n$.   We picture partitions and $r$-partitions as \emph{Young diagrams}: collections of boxes stacked in a corner, as in \eqref{tableau} (but without the numbers).  A \emph{tableau} on an $r$-partition $\lambda^\bullet$ is a function $T$ from the boxes of $\lambda^\bullet$ to the integers.  A \emph{standard tableau} on an $r$-partition $\lambda^\bullet$ of $n$ is a bijection from the boxes of $\lambda^\bullet$ to $\{1,2,\dots,n\}$ such that the enties in each $\lambda^i$ are strictly increasing left to right and top to bottom.  An example of a standard tableau on the $2$-partition $\lambda^\bullet=((3,2),(2,2))$ of $9$ is
\begin{equation} \label{tableau}
\left(\begin{ytableau}
2 & 4& 6 \\ 3 & 9 
\end{ytableau} \  , \ 
\begin{ytableau}
1 & 5 \\ 7 & 8 
\end{ytableau} \right).
\end{equation} 

Given a box $b \in \lambda^\bullet$, define $\beta(b)=l$ if $b \in \lambda^l$ and $\mathrm{ct}(b)=j-i$ if $b$ is in the $i$th row and $j$th column of $\lambda^l$.  For the example \eqref{tableau} we have $\beta(T^{-1}(5))=1$ and $\mathrm{ct}(T^{-1}(5))=1$.  

\subsection{}  Let $W=G(r,1,n)$.  The \emph{Jucys-Murphy elements} of the group algebra $\CC W$ are 
\begin{equation}
\phi_i=\sum_{\substack{1 \leq j < i \\ 0 \leq l \leq r-1}} \zeta_i^l s_{ij} \zeta_i^{-l} \quad \hbox{for $1 \leq i \leq n$.}
\end{equation}  Together with the elements $\zeta_i$ of $W$ they generate a subalgebra of $\CC W$ that acts diagonalizably on every $W$-module.  There is a bijection $\lambda^\bullet \mapsto S^{\lambda^\bullet}$ from the set of $r$-partitions of $n$ to the set of irreducible $W$-modules such that $S^{\lambda^\bullet}$ has a basis $v_T$ indexed by standard Young tableaux $T$ on $\lambda^\bullet$, and $v_T$ is determined up to scalars by the equations
\begin{equation}
\phi_i v_T=r \mathrm{ct}(T^{-1}(i)) v_T \quad \text{and} \quad \zeta_i v_T=\zeta^{\beta(T^{-1}(i))} v_T \quad \hbox{for $1 \leq i \leq n$.}
\end{equation}  We fix a $W$-invariant positive definite Hermitian form on each $S^{\lambda^\bullet}$ and assume that the norm of $v_T$ with respect to this form is $1$.  For the groups $G(2,1,n)$, it seems these versions of Jucys-Murphy elements were first written down in \cite{Che2}.

\subsection{}  \label{t def} As in \cite{DuOp} Definition 3.7, we put 
\begin{equation}
z_i=y_i x_i+c_0 \phi_i \quad \hbox{for $1\leq i \leq n$.}
\end{equation}  Together with the elements $\zeta_i$ they generate a commutative algebra $\ttt$ of $H_c$, and Theorem 5.1 of \cite{Gri2} states that $\ttt$ acts on each standard module $\Delta_c(\lambda^\bullet)$ in an upper-triangular fashion.

\subsection{} \label{gamma subsection}    We define a partial order on the boxes of $\lambda^\bullet$ by: $b \leq b'$ if $T(b)<T(b')$ for all standard Young tableaux $T$ on $\lambda^\bullet$.  Thus $b \leq b'$ if and only if $\beta(b) =\beta(b')$ and $b$ is (weakly) up and to the left of $b'$.  

We write $\Gamma=\Gamma(\lambda^\bullet)$ for the set of pairs $(P,Q)$ of tableaux on $\lambda^\bullet$ such that $P$ is a bijection from the boxes of $\lambda^\bullet$ to the set $\{1,2,\dots,n\}$, $Q$ is a filling of the boxes of $\lambda^\bullet$ by non-negative integers such that  if $b<b'$ then $Q(b) \leq Q(b')$, with  $Q(b)=Q(b')$ implying $P(b)>P(b')$.  Then Theorem 5.1 of \cite{Gri2} implies that there is a $\ttt$-eigenbasis $f_{P,Q}$ of $\Delta(\lambda^\bullet)$ such that 
$$\zeta_i f_{P,Q}=\zeta^{\beta(P^{-1}(i))} f_{P,Q}$$ and
$$z_i f_{P,Q}=(Q (P^{-1}(i))+1-(d_{\beta(P^{-1} (i))}-d_{\beta(P^{-1} (i))-Q(P^{-1}(i))-1})-r \mathrm{ct}(P^{-1}(i)) c_0) f_{P,Q}.$$  Then $f_{P,Q}$ is a polynomial function on $\CC^n$ with values in $S^{\lambda^\bullet}$, which we will refer to a a (non-symmetric) \emph{Specht-valued Jack polynomial}. The indexing here is related to that in the paper \cite{Gri2} as follows: for a pair $(\mu,T)$ consisting of a standard Young tableau $T$ on $\lambda^\bullet$ and $\mu \in \ZZ_{\geq 0}^n$, we let $w_\mu$ be the longest element of $S_n$ such that $w_\mu(\mu)$ is non-decreasing and define $P=w_\mu^{-1} T$ and $Q(b)=\mu_{P(b)}$.  Note that we used an unorthodox convention for standard Young tableaux in \cite{Gri2}, regarding them as functions from $\{1,2,\dots,n\}$ to the boxes of $\lambda^\bullet$ (we have now come to our senses).  This gives a bijection from the set of pairs $(\mu,T)$ as above to $\Gamma$. 

\section{The trigonometric presentation of $H$}

Here we give another presentation of $H$ which is adapted to the
application of intertwining operators to the classification of diagonalizable modules in the next section.

\subsection{The affine Weyl semigroup} \label{awsg}  Let $W_{\geq 0}=\ZZ_{\geq 0}^n
\rtimes S_n$.  It contains the elements $s_1,\dots,s_{n-1}$ and
$\Phi=\epsilon_n s_{n-1} \cdots s_2 s_1$, so that $s_1,\dots,s_{n-1}$
satisfy the usual Coxeter relations, and interact with $\Phi$ via the
relations
\begin{equation} \label{phi relations}
\Phi s_i=s_{i-1} \Phi \ \hbox{for $2 \leq i \leq n-1$ and} \ \Phi^2
s_1=s_{n-1} \Phi^2.
\end{equation}  In fact, the abstract semigroup with generators
$s_1,\dots,s_{n-1}$ and $\Phi$, together with the Coxeter relations
and \eqref{phi relations} is isomorphic to $W_{\geq 0}$, as we now
sketch.  

Letting $G$ be this semigroup, it follows that there is a map
$G \rightarrow W_{\geq 0}$ and thus that $s_1,\dots,s_{n-1}$ generate
a copy of $S_n$ inside $G$.  Define $\epsilon_n=\Phi s_1 \cdots
s_{n-1}$.  The relations in \eqref{phi relations}
imply that $s_i \epsilon_n=\epsilon_n s_i$ for $1 \leq i \leq n-2$,
and therefore we may unambiguously define $\epsilon_i=w \epsilon_n
w^{-1}$ for each $1 \leq i \leq n-1$ and any $w \in S_n$ with
$w(n)=i$.  It follows from this definition that $w \epsilon_i
w^{-1}=\epsilon_{w(i)}$ for all $1 \leq i \leq n$ and $w \in S_n$.
Again using \eqref{phi relations}, a direct calculation shows that
$\epsilon_n \epsilon_1=\epsilon_1 \epsilon_n$, and hence for all $1
\leq i \neq j \leq n$ choosing $w \in S_n$ with $w(1)=i$ and $w(n)=j$
gives $\epsilon_i \epsilon_j=w \epsilon_1 \epsilon_n w^{-1}=w
\epsilon_n \epsilon_1 w^{-1}=\epsilon_j \epsilon_i$.  It follows from
this that there is a map $W_{\geq 0} \rightarrow G$ inverse to the
previous one.

\subsection{The Dunkl-Opdam subalgebra}

The \emph{Dunkl-Opdam} subalgebra $\ttt$ of $H$ is the (commutative, as proved in \cite{DuOp})
subalgebra of $H$ generated by $z_1,\dots,z_n$ and
$\zeta_1,\dots,\zeta_n$.  By the PBW theorem it is isomorphic to the
polynomial ring in the variables $z_1,\dots,z_n$ tensored with the
group algebra of $(\ZZ/r \ZZ)^n$.  Define an automorphism $\phi$ of $\ttt$ by
\begin{equation}
\phi(\zeta_i)=\zeta_{i+1} \quad \hbox{for $1 \leq i \leq n-1$ and}
\quad
\phi(\zeta_n)=\zeta^{-1} \zeta_1
\end{equation} and
\begin{equation}
\phi(z_i)=z_{i+1} \quad \hbox{for $1 \leq i \leq n-1$ and}
\quad
\phi(z_n)=z_1+1-\sum_{j=0}^{r-1} (d_{j-1}-d_{j-2}) e_{1j}.
\end{equation}  Put $\Phi=x_n s_{n-1} \cdots s_1$ and $\Psi=y_1 s_1
\cdots s_{n-1}$.  By Proposition 4.3 and Lemma 5.3 of \cite{Gri1}, 
\begin{equation} \label{phipsit}
f \Phi=\Phi \phi(f) \quad \text{and} \quad f \Psi=\Psi \phi^{-1}(f)
\end{equation} for all $f \in \ttt$, and
\begin{equation} \label{ghrels}
z_i s_j=s_j z_i \ \hbox{if $j \neq i,i+1$, while} \quad z_i
s_i=s_i z_{i+1}-c_0 \pi_i \ \hbox{for $1 \leq i \leq n-1$,}
\end{equation} where
$$\pi_i=\sum_{l=0}^{r-1} \zeta_i^l \zeta_{i+1}^{-l}.$$  Thus (as observed in \cite{Dez1} and \cite{Dez2}) the
subalgebra $H_{\mathrm{gr}}$ of $H$ generated by $\ttt$ and $W=G(r,1,n)$ is isomorphic
to the generalized graded affine Hecke algebra for $G(r,1,n)$ defined
in \cite{RaSh}.  The structure of an $H_{\mathrm{gr}}$-module may be
put on a vector space by definining an action of $\ttt$ together with
operators $s_i$ satisfying the Coxeter relations together with
\eqref{ghrels} and $w \zeta_i=\zeta_{w(i)} w$ for all $w \in S_n$ and
$1 \leq i \leq n$.

\subsection{The trigonometric presentation}

$H$ contains the commutative subalgebra $\ttt$, elements $\Phi$,
$\Psi$, and $s_1,\dots,s_{n-1}$.  These satisfy the following
relations: (1) $\ttt$ and $s_1,\dots,s_{n-1}$ generate a graded affine
Hecke algebra $H_{\mathrm{gr}}$ for the group $G(r,1,n)$ inside $H$, (2) $\Phi$ and
$s_1, \dots, s_{n-1}$ generate an affine Weyl semigroup inside $H$,
(3) $\Psi$ and $s_1,\dots,s_{n-1}$ generate an affine Weyl semigroup
inside $H$ with relations
\begin{equation} \label{psi relations}
\Psi s_i=s_{i+1} \Psi \ \hbox{for $1 \leq i \leq n-2$ and} \ \Psi^2
s_{n-1}=s_1 \Psi^2,
\end{equation} and the following relations hold:
\begin{equation} \label{phipsizeta}
\zeta_i \Phi=\Phi \phi(\zeta_i), \quad \zeta_i \Psi=\Psi \phi^{-1}(\zeta_i),
\end{equation}
\begin{equation} \label{phipsi relations}
\Psi \Phi=z_1, \quad \Phi
\Psi=z_n-\kappa+\sum_{j=0}^{r-1} (d_j-d_{j-1}) e_{nj}, \quad
\text{and} \quad \Psi s_{n-1} \Phi=\Phi s_1 \Psi+c_0 \sum_{0 \leq l
\leq r-1} \zeta^{-l}
\zeta_1^l \zeta_n^{-l}.
\end{equation} In fact, this constitutes a presentation for $H$, as
we will see in the remainder of this section.  Constructing an $H$-module may therefore be done as
follows: construct an $H_{\mathrm{gr}}$-module together with operators $\Phi$ and
$\Psi$ satisfying the relations \eqref{phi relations}, \eqref{psi
relations}, \eqref{phipsizeta}, and \eqref{phipsi relations}.

Note that in case $W=S_n$ and $\kappa=0$ we have  $$\Psi \Phi-\Phi \Psi=z_1-z_n$$ and
$$(z_1-z_n) \Phi-\Phi(z_1-z_n)=$$

\begin{theorem} \label{trig pres} 
Let $A$ be the algebra generated by $H_{\mathrm{gr}}$
together with elements $\Phi$ and $\Psi$ satisfying \eqref{phi relations}, \eqref{psi
relations}, \eqref{phipsizeta}, and \eqref{phipsi
relations}.
The natural map $A \rightarrow H$ is an isomorphism.
\end{theorem}
\begin{proof}
Define elements $x_n,y_1 \in A$ by $x_n=\Phi s_1 \cdots s_{n-1}$ and
$y_1=\Psi s_{n-1} \cdots s_1$.  Then put $x_i=w x_n w^{-1}$ and $y_i=v
y_1 v^{-1}$ where $w,v \in S_n$ are chosen with $w(n)=i=v(1)$.  The
various $x_i$'s commute with one another by the discussion in \ref{awsg}, and
by symmetry the $y_i$'s commute.  Furthermore, the
algebra $A$ is generated by the group $W$ together with
$x_1,\dots,x_n$ and $y_1,\dots,y_n$.  We will show that it is spanned
by the set of all words $x_1^{a_1}\cdots x_n^{a_n} y_1^{b_1} \cdots
y_n^{b_n} w$ with $a_i,b_i \in \ZZ_{\geq 0}$ and $w \in W$.  From this
together with the PBW theorem for $H$ it will follow that the natural
map from $A$ to $H$ is an isomorphism.

It suffices to show that the span of the set of words as above is closed under
left multiplication by $x_i$'s, $y_i$'s, and $w$'s.  This is clear for
$x_i$'s and easy for $w$'s.  We will show how to reorder a product
$y_i x_j$.  First observe that by the definitions of $x_1$,$y_1$ and
the relation $\Psi \Phi=z_1$, $$y_1 x_1=\Psi \Phi=z_1.$$  By using the
graded Hecke algebra relations between $z_i$ and $s_i$ it follows
by induction on $i$ that
$$z_i=y_i x_i+a_i \quad \hbox{for some $a_i \in \CC W$.}$$ In particular 
\begin{equation}
y_n x_n+a_n=z_n=\Phi \Psi+\kappa-\sum (d_j-d_{j-1}) e_{nj}=x_n y_n +\kappa-\sum (d_j-d_{j-1}) e_{nj}.
\end{equation}  This proves that $y_n x_n=x_n y_n+b_n$ for some $b_n
\in \CC W$.  Conjugating by some $w \in S_n$ with $w(n)=i$ gives $y_i x_i=x_i
y_i+b_i$ for some $b_i \in \CC W$.   Using the last relation in \eqref{phipsi relations}, the
relations \eqref{phi relations} and \eqref{psi relations} and
the definitions of $y_1=\Psi s_{n-1} \cdots s_1$ and $x_n=\Phi s_1 \cdots s_{n-1}$ allows
one to rewrite $y_1 x_n=x_n y_1+b_{1n}$ for some $b_{1n} \in \CC W$,
and conjugating by $w \in S_n$ with $w(1)=i$ and $w(n)=j$ gives $y_i
x_j=x_j y_i+b_{ij}$ for some $b_{ij} \in \CC W$, finishing the proof.

\end{proof} 

\section{Specht-valued Jack polynomials}

For $\mu, \nu \in \ZZ_{\geq 0}^n$, write $\mu > \nu$ if either
$\mu_+ >_d \nu_+$, where $\mu_+$ and $\nu_+$ are the partition
rearrangements of $\mu$ and $\nu$ and $>_d$ denotes dominance order,
or $\mu_+=\nu_+$ and $w_\mu > w_\nu$ in Bruhat order.  Extend this to
a partial order on pairs $(\mu,T)$ by ignoring $T$: thus $(\mu,T) \geq
(\nu,S)$ exactly if $\mu \geq \nu$.  The following is Theorem~5.1 of
\cite{Gri2}; the polynomials it constructs are
$S^{\lambda^\bullet}$-valued generalizations of non-symmetric Jack polynomials. We use them to construct bases for the irreducible unitary representations in $\OO_c$.

\begin{theorem} \label{Upper triangular}
Let $\lambda$ be an $r$-partition of $n$, $\mu \in \ZZ_{\geq 0}^n$, and let $T$ be a standard tableau on $\lambda$.  Put $v_T^\mu=w_\mu^{-1}.v_T$ and recall the definitions of $\beta$ and ct given in \ref{partitions subsection}.
\begin{enumerate}
\item[(a)]  The action of $\zeta_i$ and $z_i$ on $\Delta(\lambda^\bullet)$ are given by
\begin{align*}
\zeta_i.x^\mu v_T^\mu=\zeta^{\beta(T^{-1}(w_\mu(i)))-\mu_i} x^\mu v_T^\mu
\end{align*}
and
\begin{align*} \label{z spectrum}
z_i.x^\mu v_T^\mu&=\left(  \mu_i+1-(d_{\beta(T^{-1}(w_\mu(i)))}-d_{\beta(T^{-1}(w_\mu(i)))-\mu_i-1})-c_0 r \text{ct}(T^{-1}(w_\mu(i))) \right) x^\mu v_T^\mu \\
&+\sum_{(\nu,S) < (\mu,T)} c_{\nu,S} x^\nu v_{S}^\nu.
\end{align*}
\item[(b)]  Assuming that scalars are extended to $F=\CC(c_0,d_1,d_2,\dots,d_{r-1})$, for each $\mu \in \ZZ_{\geq 0}^n$ and $T \in \text{SYT}(\lambda)$ there exists a unique $\ttt$ eigenvector $f_{\mu,T} \in \Delta(\lambda^\bullet)$ such that
\begin{equation*}
f_{\mu,T}=x^\mu v_T^\mu+\text{lower terms}.
\end{equation*}  
\end{enumerate}
The $\ttt$-eigenvalue of $f_{\mu,T}$ is determined by the formulas in part (a).
\end{theorem}

We will also index these non-symmetric Jack polynomials $f_{P,Q}$, with $(P,Q) \in \Gamma$ as in \ref{gamma subsection}.

\subsection{Intertwiners}

The intertwiners $\sigma_i$ are defined, for $1 \leq i \leq n-1$, by
\begin{equation}
\sigma_i=s_i+\frac{c_0}{z_i-z_{i+1}} \pi_i.
\end{equation}  Thus $\sigma_i$ is well-defined on any $\ttt$-weight
space on which $\pi_i$ acts by $0$ or on which $z_i$ and $z_{i+1}$
have distinct eigenvalues.

For convenience, we reproduce here Lemma 5.3 of \cite{Gri2}, which
describes how the intertwiners act on the basis $f_{\mu,T}$ of
$\Delta(\lambda^\bullet)$.  For $\mu \in \ZZ^n$ define
\begin{equation}
\phi(\mu_1,\dots,\mu_n)=(\mu_2,\mu_3,\dots,\mu_n,\mu_1+1) \quad 
\text{and} \quad \psi(\mu_1,\dots,\mu_n)=\phi^{-1}(\mu_1,\dots,\mu_n).
\end{equation}

\begin{lemma} \label{action lemma 1}
Let $\mu \in \ZZ_{\geq 0}^n$ and let $T$ be a standard Young tableau on $\lambda$.
\begin{enumerate}
\item[(a)] Suppose $\mu_i \neq \mu_{i+1}$.  If $\mu_i<\mu_{i+1}$ or $\mu_i-\mu_{i+1} \neq \beta(T^{-1}(w_\mu(i)))-\beta(T^{-1}(w_\mu(i+1)))$ mod $r$ then
\begin{equation*}
\sigma_i.f_{\mu,T}=f_{s_i.\mu,T}.
\end{equation*}
\item[(b)]  If $\mu_i>\mu_{i+1}$ and $\mu_i-\mu_{i+1} =\beta(T^{-1}(w_\mu(i)))-\beta(T^{-1}(w_\mu(i+1)))$ mod $r$ then
\begin{equation*}
\sigma_i.f_{\mu,T}=\frac{(\delta-rc_0)(\delta+rc_0)}{\delta^2} f_{s_i \mu,T},
\end{equation*} where
\begin{equation*}
\delta=\kappa(\mu_{i}-\mu_{i+1})-(d_{\beta(T^{-1}(w_\mu(i)))}-d_{\beta(T^{-1}(w_\mu(i+1)))})-c_0 r (\text{ct}(T^{-1}(w_\mu(i)))-\text{ct}(T^{-1}(w_\mu(i+1)))).
\end{equation*}
\item[(c)] Put $j=w_\mu(i)$.  If $\mu_i=\mu_{i+1}$ then
\begin{equation*}
\sigma_i.f_{\mu,T}=\begin{cases}  0 \quad &\hbox{if $s_{j-1}.T$ is not a standard tableau,} \\
f_{\mu,s_{j-1}.T} \quad &\hbox{if $\zeta^{\beta(T(j))}\neq \zeta^{\beta(T(j-1))}$,} \\ 
\left(1-\left( \frac{1}{\text{ct}(T(j-1))-\text{ct}(T(j))} \right)^2 \right)^{1/2} f_{\mu,s_{j-1}.T} \quad &\text{else.} \end{cases}
\end{equation*} 
\item[(d)]  For all $\mu \in \ZZ_{\geq 0}^n$,
\begin{equation*}
\Phi.f_{\mu,T}=f_{\phi.\mu,T}.
\end{equation*}
\item[(e)]  For all $\mu \in \ZZ_{\geq 0}^n$,
\begin{equation*}
\Psi.f_{\mu,T}=\begin{cases}  \left(\kappa \mu_n-(d_{\beta(T^{-1}(w_\mu(n)))}-d_{\beta(T^{-1}(w_\mu(n)))-\mu_n})-r \text{ct}(T^{-1}(w_\mu(n))) c_0 \right) f_{\psi.\mu,T}\quad &\hbox{if $\mu_n>0$,} \\ 0 \quad &\hbox{if $\mu_n=0$.} \end{cases}
\end{equation*}
\end{enumerate}
\end{lemma}

We define $\phi(P,Q)$ as follows: $\phi$ cycles the entries of $P$, relacing $P(b)$ by $P(b)-1$ for $P(b)>1$ and replacing $P(b)$ by $n$ if $P(b)=1$, and adds $1$ to the entry of $Q$ in the box that $P$ labels with $1$. Thus for example if $(P,Q)$ is the pair
$$\begin{ytableau}
3 & 2 \\ 1 & 4
\end{ytableau} \ , \ 
\begin{ytableau}
0 & 0 \\ 0 & 1 
\end{ytableau}
$$
then $\phi(P,Q)$ is the pair
$$\begin{ytableau}
2 & 1 \\ 4 & 3
\end{ytableau} \ , \ 
\begin{ytableau}
0 & 0 \\ 1 & 1
\end{ytableau}
$$ We write $\psi=\phi^{-1}$ for the inverse of $\phi$ (which does not preserve $\Gamma(\lambda)$. 

The lemma, translated into the $(P,Q)$-indexing and using $c(b)=d_{\beta(b)}+r \mathrm{ct}(b) c_0$ is
\begin{lemma} \label{action lemma}
Let $\mu \in \ZZ_{\geq 0}^n$ and let $T$ be a standard Young tableau on $\lambda$.
\begin{enumerate}
\item[(a)]  If $Q(P^{-1}(i)) < Q(P^{-1}(i+1))$ or $Q(P^{-1}(i))-Q(P^{-1}(i+1)) \neq \beta(P^{-1}(i))-\beta(P^{-1}(i+1))$ mod $r$ then
\begin{equation*}
\sigma_i.f_{P,Q}=f_{s_i P,Q}.
\end{equation*}
\item[(b)]  If $Q(P^{-1}(i))> Q(P^{-1}(i+1))$ and $Q(P^{-1}(i))-Q(P^{-1}(i+1)) =\beta(P^{-1}(i))-\beta(P^{-1}(i+1))$ mod $r$ then
\begin{equation*}
\sigma_i.f_{P,Q}=\frac{(\delta-rc_0)(\delta+rc_0)}{\delta^2} f_{s_i P,Q},
\end{equation*} where
\begin{equation*}
\delta=Q(P^{-1}(i))-Q(P^{-1}(i+1))-(c(P^{-1}(i))-c(P^{-1}(i+1))).
\end{equation*}
\item[(c)]   If $Q(P^{-1}(i))i=Q(P^{-1}(i+1))$ then
\begin{equation*}
\sigma_i.f_{P,Q}=\begin{cases}  0 \quad &\hbox{if $(s_iP,Q)$ is not an element of $\Gamma(\lambda)$,} \\
f_{s_iP,Q} \quad &\hbox{if $\beta(P^{-1}(i)) \neq \beta(P^{-1}(i+1))$,} \\ 
\left(1-\left( \frac{1}{\text{ct}(P^{-1}(i+1))-\text{ct}(P^{-1}(i))} \right)^2 \right)^{1/2} f_{s_iP,Q} \quad &\text{else.} \end{cases}
\end{equation*} 
\item[(d)]  For all $(P,Q) \in \Gamma(\lambda)$,
\begin{equation*}
\Phi.f_{P,Q}=f_{\phi(P,Q)}.
\end{equation*}
\item[(e)]  For all $(P,Q) \in \Gamma(\lambda)$,
\begin{equation*}
\Psi.f_{P,Q}=\begin{cases}  \left(Q(P^{-1}(n))-c(P^{-1}(n))+d_{\beta(P^{-1}(n))-Q(P^{-1}(n))} \right) f_{\psi(P,Q)}\quad &\hbox{if $Q(P^{-1}(n))>0$,} \\ 0 \quad &\hbox{if $Q(P^{-1}(n))=0$.} \end{cases}
\end{equation*}
\end{enumerate}
\end{lemma}

\section{Diagonalizability}

\subsection{Weight spaces}

Fix an $r$-partition $\lambda^\bullet$ and let $\Gamma=\ZZ_{\geq 0}
\times \mathrm{SYT}(\lambda^\bullet)$ (via the bijection of \ref{gamma subsection} this is the same $\Gamma$ as defined there).  Given
$c=(c_0,d_0,\dots,d_{r-1}) \in \CC^{r+1}$, $(\mu,T) \in \Gamma$ and
$1 \leq i \leq n$, write $\mathrm{wt}_c(\mu,T)_i$ for the pair
\begin{equation} \label{i weight}
\mathrm{wt}_c(\mu,T)_i=(\mu_i+1-(d_{\beta(T^{-1} w_\mu(i))}-d_{\beta(T^{-1}
w_\mu(i))-\mu_i-1})-r \mathrm{ct}(T^{-1} w_\mu(i)) c_0,
\zeta^{\beta(T^{-1} w_\mu(i))-\mu_i})
\end{equation}  Then define $(\mu,T)$ to be \emph{$c$-folded} (or
simply \emph{folded} when $c$ is fixed or clear from context) if
$\mathrm{wt}_c(\mu,T)_i=\mathrm{wt}_c(\mu,T)_{i+1}$ for some $1 \leq i
\leq n-1$.  Foldings create non-trivial Jordan blocks:
\begin{lemma} \label{foldinglemma}
Suppose that $\mathrm{wt}_c(\mu,T)_i=\mathrm{wt}_c(\mu,T)_{i+1}$ and write
$\alpha=\mathrm{wt}_c(\mu,T)_i$.  Put $f_1=f_{\mu,T}$ and $f_2=s_i
f_1$.  If $z_i f_1=\alpha f_1=z_{i+1} f_1$, then $(z_i-\alpha)
f_2=-rc_0 f_1$ and $(z_{i+1}-\alpha) f_2=rc_0 f_1$.
\end{lemma}
\begin{proof}
Apply \eqref{ghrels}.
\end{proof}

As in \cite{Gri2}, for a box $b \in \lambda^\bullet$ and a positive integer $k$ define a set 
\begin{equation}
\Gamma_{b,k}=\{(\mu,T) \in \Gamma \ | \ \mu^-_{T(b)} \geq k \},
\end{equation} and for an ordered pair of distinct boxes $b_1,b_2 \in
\Gamma$ and a positive integer $k \in \ZZ_{>0}$, define the subset $\Gamma_{b_1,b_2,k}$ of $\Gamma$ by
\begin{align}
(\mu,T) \in \Gamma_{b_1,b_2,k} \quad \iff \quad &\hbox{either $
\mu^-_{T(b_1)} - \mu^-_{T(b_2)} > k$} \notag \\ 
& \hbox{or  $\mu^-_{T(b_1)} - \mu^-_{T(b_2)}=k$ and $w_\mu^{-1}(T(b_1)) < w_\mu^{-1}(T(b_2))$}.
\end{align}  Via the bijection with pairs $(P,Q)$ as above, these definitions
become somewhat easier on the eyes:
\begin{equation}
\Gamma_{b,k}=\{(P,Q) \ | \ Q(b) \geq k \}
\end{equation} and
\begin{equation}
\Gamma_{b_1,b_2,k}=\{(P,Q) \ | \ Q(b_1)-Q(b_2) > k, \ \mathrm{or} \
Q(b_1)-Q(b_2)=k \ \mathrm{and} \ P(b_1) < P(b_2)
\end{equation}

For a given parameter $c$,
define the set $\Gamma_c \subseteq \Gamma$ by 
\begin{equation}
\Gamma_c=\bigcap_{b,k} \Gamma_{b,k}^c \cap \bigcap_{b_1,b_2,k} \Gamma_{b_1,b_2,k}^c,
\end{equation} where for a subset $X \subseteq \Gamma$ we write $X^c$
for its complement, the first intersection runs over pairs $b \in
\lambda^\bullet$ and $k \in \ZZ_{>0}$ such that 
$$k=d_{\beta(b)}-d_{\beta(b)-k}+r \mathrm{ct}(b) c_0$$ and the second
intersection runs over triples $b_1,b_2 \in \lambda^\bullet$, $k \in
\ZZ_{>0}$ such that $k=\beta(b_1)-\beta(b_2) \ \mathrm{mod} \ r$ and
$$k=d_{\beta(b_1)}-d_{\beta(b_2)}+r(\mathrm{ct}(b_1)-\mathrm{ct}(b_2)
\pm 1) c_0.$$  The motivation for the definition is that the set $\Gamma_c$ contains exactly those $(\mu,T)$ such that $f_{\mu,T}$ may be constructed from some $v_T \in S^{\lambda^\bullet}$ by applying a sequence of \emph{invertible} intertwining operators; this is a consequence of Lemma 7.4 of \cite{Gri2}.

The definition of $\Gamma_c$ may be rephrased in terms of pairs
$(P,Q)$ as follows: a pair $(P,Q)$ is in $\Gamma_c$ if and only if
the following conditions hold: 
\begin{enumerate}
\item[(a)] whenever $b \in \lambda^\bullet$
and $k \in \ZZ_{>0}$ with $k=d_{\beta(b)}-d_{\beta(b)-k}+r
\mathrm{ct}(b) c_0$ we have $Q(b)<k$, and 
\item[(b)] whenever $b_1,b_2 \in
\lambda^\bullet$ and $k \in \ZZ_{>0}$ with $\beta(b_1)-\beta(b_2)=k$ mod $r$ and
$k=d_{\beta(b_1)}-d_{\beta(b_2)}+r (\mathrm{ct}(b_1)-\mathrm{ct}(b_2)
\pm 1)c_0$ we have $Q(b_1) \leq Q(b_2)+k$, with equality implying
$P(b_1)>P(b_2)$. \end{enumerate}

The \emph{boundary} of $\Gamma_c$ is 
\begin{equation}
\partial \Gamma_c=\{(\mu,T) \in \Gamma-\Gamma_c \ | \ (\psi.\mu,T) \in
\Gamma_c \ \mathrm{or} \ (s_i.\mu,T) \in \Gamma_c \ \hbox{for some $1
\leq i \leq n$} \}.
\end{equation}
\begin{lemma} \label{distinct weights}  Assume $c_0 \neq 0$.
\begin{enumerate}
\item[(a)] Suppose $(\mu,T) \in \Gamma_c$ and $(\nu,S) \in \Gamma$
with $\text{wt}_c(\mu,T)=\text{wt}_c(\nu,S)$.  Then $(\nu,S)=(\mu,T)$.
\item[(b)] For all $(\mu,T) \in \Gamma_c$ and $1 \leq i \leq n-1$, we
have $\mathrm{wt}_c(\mu,T)_i \neq \mathrm{wt}_c(\mu,T)_{i+1}$.  
\item[(c)] The non-symmetric generalized Jack polynomials $f_{\mu,T}$
for $(\mu,T) \in \Gamma_c \cup \partial \Gamma_c$ are all well-defined at $c$.  
\item[(d)] If $(\mu,T) \in \partial \Gamma_c$ is folded then
$L_c(\lambda^\bullet)$ is not $\ttt$-diagonalizable.
\end{enumerate}
\end{lemma}
\begin{proof}
The definition of $\Gamma_c$ and Lemma 7.4 of \cite{Gri2} together
imply that the intertwiners connecting different $\ttt$-weight spaces indexed
by $\Gamma_c$ are all invertible; it follows that every such weight
space has the same dimension.  Since the weight spaces in degree $0$
(coming from $S^{\lambda^\bullet}$) are all one dimensional (here we use $c_0 \neq 0$), this
proves (a).  Part (b) follows from (a) together with
Lemma~\ref{foldinglemma}.  By part (b) the intertwining operators are
well-defined on all weight spaces coming from $\Gamma_c$; this allows one to
recursively construct all Jack polynomials coming from $\Gamma_c \cup \partial
\Gamma_c$ recursively, proving (c).

Now we prove (d).  Suppose $(\mu,T) \in \partial \Gamma_c$ is folded
and let $f_1=f_{\mu,T}$.  If
$\mathrm{wt}_c(\mu,T)_i=\mathrm{wt}_c(\mu,T)_{i+1}$ then part (b)
implies $z_i f=\alpha f=z_{i+1} f$ for some $\alpha \in \CC$, and by
Lemma~\ref{foldinglemma} $f_2=s_i f_1$ witnesses a non-trivial Jordan
block for $\ttt$: $(z_i-\alpha) f_2=-r c_0 f_1=-(z_{i+1}-\alpha)f_2$.
We will show that the image of $f_2$ in $L_c(\lambda^\bullet)$ is
non-zero.  By (b) we must have either $(s_{i-1} \mu,T) \in \Gamma_c$
or $(s_{i+1}\mu,T) \in \Gamma_c$ or $(\psi \mu,T) \in \Gamma_c$.  

Suppose that $(\psi \mu,T) \in \Gamma_c$.  It follows that the map
$\Psi$ is not an injection on the weight space for $(\mu,T)$, and
hence by the second equation in \eqref{phipsi relations} the
$z_n$-eigenvalue on $(\mu,T)$ is given by
$\alpha=1-(d_{\beta_j}-d_{\beta_j-1})$ where $\zeta_n
f_{\mu,T}=\zeta^{\beta_j} f_{\mu,T}$.  Compute using \eqref{phipsi
relations} and Lemma~\ref{foldinglemma}
$$\Phi \Psi f_2=(z_n-1+\sum_{j=0}^{r-1} (d_j-d_{j-1})
e_{nj})f_2=(z_n-\alpha) f_2=rc_0 f_1.$$ This equation implies that $\Psi f_2=a f_{(\psi \mu,T)}$
for some $a \in \CC^\times$ and since the image of $f_{(\psi \mu,T)}$
in $L_c(\lambda^\bullet)$ is non-zero, so is the image of $f_2$. \end{proof}

Now we can given our first (not completely explicit) description of
the diagonalizable $L_c(\lambda^\bullet)$'s.  When $c_0=0$ the modules $\Delta_c(\lambda^\bullet)$ are all diagonalizable (but with weight spaces of dimension greater than $1$), so the following theorem finishes the classification.  

\begin{theorem} \label{diag char1}
Suppose $c_0 \neq 0$. The module $L_c(\lambda^\bullet)$ is diagonalizable exactly if no
element of $\partial \Gamma_c$ is folded; in this case a basis is
given by $\{f_{\mu,T} \ | \ (\mu,T) \in \Gamma_c \}$.  
\end{theorem}
\begin{proof}
Given Lemma~\ref{distinct weights}, it remains to show that if no
element of $\partial \Gamma_c$ is folded then $L_c(\lambda^\bullet)$
is diagonalizable with the given basis.  Let $V$ be the abstract
$\CC$-vector space with basis given by $\{f_{\mu,T} \ | \ (\mu,T) \in
\Gamma_c \}$, and define actions of $\ttt$,
$\sigma_1,\dots,\sigma_{n-1}$, $\Phi$, and $\Psi$ on $V$ as follows:
the $\ttt$-action has $f_{\mu,T}$ as eigenfunctions with eigenvalues
given by Theorem~\ref{Upper triangular}, and the action of $\sigma_i$,
$\Phi$, and $\Psi$ is given by the formulas in Lemma~\ref{action
lemma} with the following exceptions: if $(s_i \mu,T) \notin \Gamma_c$
then we put $\sigma_i f_{\mu,T}=0$ and if $(\phi \mu,T) \notin
\Gamma_c$ then we put $\Phi f_{\mu,T}=0$.  Without using the
hypothesis that no element of $\partial \Gamma_c$ is folded, it
follows from these definitions that the $\sigma_i$'s satisfy the braid
relations, that 
\begin{equation} \label{e1}
\sigma_i^2=\frac{(z_i-z_{i+1}-c_0 \pi_i)(z_i-z_{i+1}+c_0
\pi_i)}{(z_i-z_{i+1})^2},
\end{equation} that
\begin{equation} \label{e2}
f \sigma_i=\sigma_i s_i(f) \quad
\hbox{for $1 \leq i \leq n-1$ and $f \in \ttt$,}
\end{equation} that
\begin{equation}
\Phi \sigma_i= \sigma_{i-1} \Phi \quad \hbox{for $2 \leq i \leq n-1$
and} \quad \Phi^2 \sigma_1=\sigma_{n-1} \Phi^2,
\end{equation} that
\begin{equation}
\Psi \sigma_i=\sigma_{i+1} \Psi \quad \hbox{for $1 \leq i \leq n-2$
and} \quad \Psi^2 s_{n-1}=s_1 \Psi^2,
\end{equation} that
\begin{equation}
f \Phi=\Phi \phi(f) \quad \text{and} \quad f \Psi=\Psi \phi^{-1}(f)
\quad \hbox{for $f \in \ttt$}
\end{equation} and that
\begin{equation}
\Psi \Phi=z_1, \quad \Phi
\Psi=z_n-\kappa+\sum_{j=0}^{r-1} (d_j-d_{j-1}) e_{nj}, \quad
\text{and} \quad \Psi \sigma_{n-1} \Phi=\Phi \sigma_1 \Psi.
\end{equation}  Indeed, these formulas hold by Lemma~\ref{action
lemma} when applied to those $f_{\mu,T}$ for which the result stays in
$\Gamma_c$ at each stage; some care must be taken near the boundary, as we indicate next.  

We will sketch a check of the relation 
$\Psi \sigma_{n-1} \Phi=\Phi \sigma_1 \Psi$ here; the others involve 
similar reasoning.  Suppose first that $\Phi f_{P,Q}=0$, that is, 
$\phi(P,Q) \notin \Gamma_c$.  By definition of $\Gamma_c$ and $\phi$, 
there is a box $b$ with $Q(b)=k-1$, $P(b)=1$ and such that if $Q'(b) \geq k$ for some
$(P',Q') \in \Gamma$ then $(P',Q') \notin \Gamma_c$.  Now observe that
$\phi s_1 \psi (P,Q)=(P',Q')$ with $Q'(b)=k$ and hence $(P',Q') \notin \Gamma_c$, so 
both operators act by zero on $f_{P,Q}$.  

If $\phi(P,Q) \in \Gamma_c$ but 
$s_{n-1} \phi (P,Q) \notin \Gamma_c$, then writing 
$(P',Q')=\phi(P,Q)$, setting $b_1=P'^{-1}(n)$ and $b_2=P'^{-1}(n-1)$ we have $Q'(b_1)=Q' (b_2)+k$ for some 
positive integer $k$, $k=\beta(b_1)- \beta(b_2) \ \mathrm{mod} \ r$, such that if $(P_0,Q_0) \in \Gamma$ with $Q_0(b_1)>Q_0(b_2)+k$, or $Q_0(b_1)=Q_0(b_2)+k$ and $P_0(b_1)<P_0(b_2)$, then $(P_0,Q_0) \notin \Gamma_c$.  It follows from this that $s_1 \psi(P,Q) \notin \Gamma_c$, so again both operators act by zero on $f_{P,Q}$.  The other cases are handled in a similar fashion.

Now define the action of $s_1,\dots,s_{n-1}$ on $V$ by the formula
\begin{equation} \label{si def}
s_i=\sigma_i-\frac{c_0}{z_i-z_{i+1}} \pi_i.
\end{equation}  This makes sense by part (b) of Lemma~\ref{distinct
weights}.  Using Theorem~\ref{trig pres} we must check that the
$s_i$'s and $\ttt$ satisfy the graded Hecke relations.  The relations
\eqref{ghrels} follow from the definition \eqref{si def} and the
relations \eqref{e2}.  The fact that $s_i^2=1$ follows from \eqref{si
def} and \eqref{e1}.  The fact that the braid relations are satisfied
will be the first place the hypothesis that $\partial
\Gamma_c$ contains no folds is used.  Compute:
\begin{align*}
s_i s_{i+1} s_i&=\left( \sigma_i-\frac{c_0}{z_i-z_{i+1}} \pi_i
\right)\left( \sigma_{i+1}-\frac{c_0}{z_{i+1}-z_{i+2}} \pi_{i+1}
\right)\left( \sigma_i-\frac{c_0}{z_i-z_{i+1}} \pi_i \right) \\
&=\sigma_i \sigma_{i+1} \sigma_i-c_0 \sigma_i \sigma_{i+1}
\frac{1}{z_i-z_{i+1}} \pi_i -c_0 \sigma_i^2 \frac{1}{z_i-z_{i+2}}\pi_{i,i+2}-c_0
\sigma_{i+1} \sigma_i \frac{1}{z_{i+1}-z_{i+2}} \pi_{i+1} \\
&+c_0^2 \sigma_i \frac{1}{(z_{i+1}-z_{i+2})(z_i-z_{i+1})} \pi_i
\pi_{i+1} + c_0^2 \sigma_{i+1} \frac{1}{(z_i-z_{i+2})(z_i-z_{i+1})}
\pi_{i,i+2} \pi_{i+1} \\ &+c_0^2 \sigma_i
\frac{1}{(z_{i+1}-z_i)(z_i-z_{i+2})} \pi_i \pi_{i,i+2}
-c_0^3 \frac{1}{(z_i-z_{i+1})^2(z_{i+1}-z_{i+2})} \pi_i \pi_{i+1} \pi_i.
\end{align*}  This preceding calculation was formal, but the
hypothesis that no element of $\partial \Gamma_c$ is folded is exactly
what is needed to ensure that the right-hand side of the above
equation is well-defined when applied to $f_{\mu,T}$ for all $(\mu,T)
\in \Gamma_c$.  Routine arithmetic verifies that it is the same as the
corresponding expression for $s_{i+1} s_i s_{i+1}$.  This verifies
that we have the structure of an $\HH_{\mathrm{gr}}$-module on $V$.  

Verification of the relations \eqref{phi relations}, \eqref{psi
relations}, and the last equation in \eqref{phipsi relations} is
exactly analogous, again using the hypothesis that no element of
$\partial \Gamma_c$ is folded: for instance, one computes
\begin{align}
\Psi s_{n-1} \Phi&=\Psi \left( \sigma_{n-1} -\frac{c_0}{z_{n-1}-z_n} \pi_{n-1,n} \right) \Phi=\Phi \sigma_1 \Psi-\Psi \Phi \phi\left( \frac{c_0}{z_{n-1}-z_n} \pi_{n-1,n}\right) \\
&=\Phi s_1 \Psi-\Phi \Psi \phi^{-1} \left( \frac{c_0}{z_{1}-z_2} \pi_{1,2} \right)-\Psi \Phi \phi\left( \frac{c_0}{z_{n-1}-z_n} \pi_{n-1,n}\right). \notag
\end{align}  The hypothesis that there are no folded elements of $\partial \Gamma_c$ implies that this last expression makes sense when applied to any $f_{P,Q}$ for $(P,Q) \in \Gamma_c$, and a straightforward calculation shows that it is equivalent to the last relation in \eqref{phipsi relations}.

We have therefore defined an $H_c$-module struture on $V$. It follows from the construction and Lemma \ref{distinct weights} that the $\ttt$-weight spaces on $V$ are all one-dimensional, and from Lemma 7.4 of \cite{Gri2} that any non-zero weight vector generates $V$ as an $H_c$-module.Thus $V$ is irreducible. It belong to category $\OO_c$ since by construction $\Psi$ (and hence each $y_i$) is locally nilpotent on it. The construction implies that its degree $0$ piece is isomorphic to $S^{\lambda^\bullet}$, and hence it is isomorphic to $L_c(\lambda^\bullet)$, which is therefore diagonalizable with basis $f_{P,Q}$ for $(P,Q) \in \Gamma_c$. 
\end{proof}  

\section{Combinatorics of folds}

\subsection{Near folds} We first obtain some limitations on the types of folds that can occur in
$\partial \Gamma_c$.  First, we switch from now on to the $(P,Q)$
notation for elements of $\Gamma$, and we define a \emph{near fold} to
be an element $(P,Q) \in \Gamma_c$ such that $\phi(P,Q)$ or $s_i
(P,Q)$ is folded for some $1 \leq i \leq n-1$ (by Lemma \ref{distinct weights} this fold is then in the boundary $\partial \Gamma_c$).  Here we define $s_i(P,Q)=(s_iP,Q)$, and $\phi(P,Q)=(P',Q')$ with
\begin{equation}
P'(b)=\begin{cases} P(b)-1 \quad &\hbox{if $P(b)>1$, and} \\ n \quad &\hbox{if $P(b)=1$,} \end{cases}
\end{equation} and
\begin{equation}
Q'(b)=\begin{cases} Q(b) \quad &\hbox{if $P(b)>1$, and} \\ Q(b)+1 \quad &\hbox{if $P(b)=1$.} \end{cases}
\end{equation}  These definitions are compatible with the corresponding ones for the $(\mu,T)$ notation via the bijection of \ref{gamma subsection}.

The \emph{upper rim} of a partition $\lambda$ is the set of boxes $b \in \lambda$ such that there is no box immediately above $b$.  The \emph{upper rim} of an $r$-partition $\lambda^\bullet$ is the union of the upper rims of its components $\lambda^l$.  The \emph{left rim} of a partition (resp. multipartition) is defined analogously as the set of boxes with no box immediately to the left.

\begin{lemma} \label{fold char}
$(P,Q) \in \Gamma_c$ is a near fold if and only if there is a positive
integer $k$ and boxes $b_1,b_2 \in \lambda^\bullet$ such that $b_1$ is
a removable box, $b_2$ is on the upper rim or left rim of
$\lambda^\bullet$, $k=\beta(b_1)-\beta(b_2) \ \mathrm{mod} \ r$, 
$$k=d_{\beta(b_1)}-d_{\beta(b_2)}+r(\mathrm{ct}(b_1)-\mathrm{ct}(b_2))c_0,$$
and either
\begin{enumerate}
\item[(a)] $\mathrm{ct}(b_2)=0$ (so $b_2$ is the upper left hand
corner of $\lambda^{\beta(b_2)}$), $Q(b_1)=k-1$, $Q(b_2)=0$,
$P(b_1)=1$, and $P(b_2)=n$, or
\item[(b)] $\mathrm{ct}(b_2) \neq 0$ and there are $a \in \ZZ_{\geq
0}$ and a box $b_3<b_2$ adjacent to $b_2$ with
$\mathrm{ct}(b_3)=\mathrm{ct}(b_2) \pm 1$
with $Q(b_1)=a+k$, $Q(b_2)=a=Q(b_3)$, $P(b_1)=i+1$, $P(b_2)=i-1$, and
$P(b_3)=i$ for some $2 \leq i \leq n-1$.
\end{enumerate}
\end{lemma}
\begin{proof}
First, if (a) holds then $\phi(P,Q)$ is folded, and if (b) holds then $s_i (P,Q)$ is folded.  This follows 
from the formula given in \ref{gamma subsection} for the $\ttt$-eigenvalue of $f_{P,Q}$,
 together with the formulas for $\phi(P,Q)$ and $s_i(P,Q)$ given above.

For the converse, assume first that $s_i(P,Q)=(s_iP,Q)$ is folded for some $(P,Q) \in
\Gamma_c$ and $1 \leq i \leq n-1$.  We will show that in this case we
are in situation (b) of the lemma.  It follows from our assumption that either 
$$Q(P^{-1}(i))-Q(P^{-1}(i+2))=d_{\beta(P^{-1}(i))}-d_{\beta(P^{-1}(i+2))}+r(\mathrm{ct}(P^{-1}(i))-\mathrm{ct}(P^{-1}(i+2)))c_0$$
with $Q(P^{-1}(i))-Q(P^{-1}(i+2))=\beta(P^{-1}(i))-\beta(P^{-1}(i+2))
\ \mathrm{mod} \ r$, or that the analogous equations, replacing $i$
and $i+2$ by $i-1$ and $i+1$, hold.   In any case there are boxes $b_1,b_2 \in \lambda^\bullet$ and a non-negative integer $k$ with
\begin{equation} \label{fold eqn}
k=d_{\beta(b_1)}-d_{\beta(b_2)}+r(\mathrm{ct}(b_1)-\mathrm{ct}(b_2))c_0,
\end{equation}
$$k=\beta(b_1)-\beta(b_2) \ \mathrm{mod} \ r, \quad Q(b_1)-Q(b_2)=k, \quad \text{and} \quad |P(b_1)-P(b_2)|=2.$$  First observe that $k > 0$
since otherwise $Q(b_1)=Q(b_2)$, $\beta(b_1)=\beta(b_2)$, and
$\mathrm{ct}(b_1)=\mathrm{ct}(b_2)$ contradicts $|P(b_1)-P(b_2)|=2$.  Since $(P,Q) \in \Gamma_c$, \eqref{fold eqn} implies $\mathrm{ct}(b_2) \neq 0$.  If $b_2$ is not on the upper or left rim of $\lambda^\bullet$, there are two boxes $b_3,b_4 < b_2$ with $\mathrm{ct}(b_3)=\mathrm{ct}(b_2) \pm 1$ and $\mathrm{ct}(b_4)=\mathrm{ct}(b_2) \pm 1$.  Now \eqref{fold eqn} together with $(P,Q) \in \Gamma_c$ implies that $Q(b_3)=Q(b_2)=Q(b_4)$ and that $P(b_1)>P(b_3),P(b_4)>P(b_2)$, contradicting $|P(b_1)-P(b_2)|=2$.  On the other hand, since $\mathrm{ct}(b_2) \neq 0$ there is always at least one box $b_3$ as above, so we have $P(b_1)=P(b_3)+1=P(b_2)+2$.  

If $b_1$ is not a removable box, then there is a box $b>b_1$ such that $\mathrm{ct}(b)=\mathrm{ct}(b_1) \pm 1$, and \eqref{fold eqn} once more implies $Q(b)=Q(b_1)$ and hence $P(b_1)>P(b)>P(b_2)$, which contradicts $P(b_1)=P(b_3)+1=P(b_2)+2$.  

Now assume $\phi(P,Q)$ is folded for some $(P,Q) \in \Gamma_c$.  Then
$$Q(P^{-1}(1))+1-Q(P^{-1}(n))=d_{\beta(P^{-1}(1))}-d_{\beta(P^{-1}(n))}+r(\mathrm{ct}(P^{-1}(1))-\mathrm{ct}(P^{-1}(n))c_0$$ and
$$Q(P^{-1}(1))+1-Q(P^{-1}(n))=\beta(P^{-1}(1))-\beta(P^{-1}(n)) \
\mathrm{mod} \ r.$$  

Assume first that $Q(P^{-1}(1))+1-Q(P^{-1}(n))<0$
and let $k=Q(P^{-1}(n))-Q(P^{-1}(1))-1$, $b_1=P^{-1}(n)$ and
$b_2=P^{-1}(1)$.  If $b_2$ is the upper left hand corner of
$\lambda^{\beta(b_2)}$ then $\mathrm{ct}(b_2)=0$ and the equations
$$k=d_{\beta(b_1)}-d_{\beta(b_2)}+r \mathrm{ct}(b_1) c_0 \quad
\text{and} \quad k=\beta(b_1)-\beta(b_2) \ \mathrm{mod} \ r$$ together
with $(P,Q) \in \Gamma_c$ force $Q(b_1) < k=Q(b_1)-Q(b_2)-1$,
contradiction.  Thus $b_2$ is not the upper left-hand corner of
$\lambda^{\beta(b_2)}$ and hence there is a box $b_3<b_2$ with
$\mathrm{ct}(b_3)=\mathrm{ct}(b_2) \pm 1$.  Now
$$k=d_{\beta(b_1)}-d_{\beta(b_3)}+r(\mathrm{ct}(b_1)-\mathrm{ct}(b_3)
\pm 1) c_0 \quad \text{and} \quad k=\beta(b_1)-\beta(b_3) \
\mathrm{mod} \ r$$ together with $(P,Q) \in \Gamma_c$ imply that
$Q(b_1)-Q(b_3) \leq k=Q(b_1)-Q(b_2)-1 \leq Q(b_1)-Q(b_3)-1$,
contradiction.  Therefore $Q(P^{-1}(n))-Q(P^{-1}(1))-1 \geq 0$.

Let $k=Q(P^{-1}(1))+1-Q(P^{-1}(n))$, $b_1=P^{-1}(n)$ and
$b_2=P^{-1}(1)$.  Then $k \geq 0$. 

If $k=0$ then
$\beta(b_1)=\beta(b_2)$, $Q(b_1)=Q(b_2)+1$,
$\mathrm{ct}(b_1)=\mathrm{ct}(b_2)$ and $Q(b_1)=Q(b_2)-1$ implying $b_1<b_2$.  But since
$P(b_1)=1$ there can be no box $b>b_1$ with $Q(b)=Q(b_1)$ and likewise
since $P(b_2)=n$ there can be no box $b<b_2$ with $Q(b)=Q(b_2)$,
contradiction.

Therefore $k>0$, and the equations
$$k=d_{\beta(b_1)}-d_{\beta(b_2)}+r(\mathrm{ct}(b_1)-\mathrm{ct}(b_2))c_0
\quad \text{and} \quad k=\beta(b_1)-\beta(b_2) \ \mathrm{mod} \ r$$
together with $(P,Q) \in \Gamma_c$, $P(b_1)=1$ and $P(b_2)=n$ imply
that there is no box $b$ with $b>b_1$ or $b<b_2$, and hence also that
$Q(b_1)=k-1$.  This proves that we are in case (a) of the lemma.
\end{proof}

\subsection{Proof of Theorem \ref{diag thm}} Now thanks to Theorem \ref{diag char1}, our first main Theorem \ref{diag thm} is a consequence of the following:

\begin{theorem}
Assume $c_0 \neq 0$.  The module $L_c(\lambda^\bullet)$ is diagonalizable if and only if for
every removable box $b \in \lambda^\bullet$, either $k_c(b)=\infty$ or $l_c(b)<k_c(b)$.
\end{theorem}
\begin{proof}
For convenience, we define a statistic $l_c'(b)$ similar to $l_c(b)$, but without restricting to outside addable boxes: more precisely, $l'_c(b)$ is the smallest integer $l$ such that either there exists a box $b' \in \lambda^{\beta(b)-l}$ with
$$l=d_{\beta(b)}-d_{\beta(b)-l}+r(\mathrm{ct}(b)-\mathrm{ct}(b') \pm 1) c_0$$ or 
$$l=d_{\beta(b)}-d_{\beta(b)-l}+r \mathrm{ct}(b) c_0.$$  We will prove that $L_c(\lambda^\bullet$ is diagonalizable if and only if for every removable box $b$ of $\lambda^\bullet$, we have $l_c'(b) < k_c(b)$; one checks that $l_c'(b) < k_c(b)$ if and only if $l_c(b) < k_c(b)$, so this will finish the proof. 

Suppose first that there is a removable box $b_1 \in \lambda^\bullet$
with $k_c(b_1)< \infty$ and $k_c(b_1) \leq l_c'(b_1)$ and write $k=k_c(b_1)$.
Then there is $b_2$ in the upper or left rim of
$\lambda^{\beta(b_1)-k}$ with
$$k=d_{\beta(b_1)}-d_{\beta(b_2)}+r(\mathrm{ct}(b_1)-\mathrm{ct}(b_2))c_0.$$
 If $b_2$ is the upper left-hand corner of
$\lambda^{\beta(b_1)-k}$ then set $Q(b_1)=k-1$ and $Q(b)=0$ for all
other $b \in \lambda^\bullet$, and put $P(b_1)=1$, $P(b_2)=n$ and then
complete $P$ to a reverse standard Young tableau on
$\lambda^\bullet$.  It follows from the assumption $k_c(b_1) \leq
l_c'(b_1)$ that all the inequalities necessary for $(P,Q) \in \Gamma_c$ are
satisfied, and therefore $(P,Q)$ is a near fold in $\Gamma_c$ so
$L_c(\lambda^\bullet)$ is not diagonalizable.  

If $\mathrm{ct}(b_2) \neq 0$, then we define $Q(b_1)=k$ and $Q(b)=0$ for
all other $b \in \lambda^\bullet$.  We now define
$P$ for which $(P,Q) \in \Gamma_c$: let $b_3$ be the box with
$b_3<b_2$ and $\mathrm{ct}(b_3)=\mathrm{ct}(b_2) \pm 1$, let $i$ be maximal so that there
exists a reverse standard Young tableau $T$ on $\lambda^\bullet$ with
$T(b_2)=i$, and define $P(b_1)=i+1$, $P(b_2)=i-1$, $P(b_3)=i$ and then
complete $P$ to a reverse standard Young tableau on
$\lambda^\bullet-\{b\}$.  It is then straightforward to check that
$(P,Q) \in \Gamma_c$ is a near fold so that $L_c(\lambda^\bullet)$ is
not diagonalizable.

Conversely, suppose that $l_c'(b)<k_c(b)$ for all corner boxes $b$ such
that $k(b)<\infty$.  By Theorem~\ref{diag char1} it suffices to show that there
are no near folds in $\Gamma_c$.  Suppose towards a contradiction that
$(P,Q) \in \Gamma_c$ is a near fold.  By Lemma~\ref{fold char} there
is a corner box $b_1 \in \lambda^\bullet$, an integer $k \in \ZZ_{>0}$, and a box $b_2$ in the upper
or left rim of $\lambda^{\beta(b_1)-k}$ with
$$k=d_{\beta(b_1)}-d_{\beta(b_2)}+r(\mathrm{ct}(b_1)-\mathrm{ct}(b_2))c_0,$$
and one of the following holds:
\begin{enumerate}
\item[(a)] The box $b_2$ is the upper left-hand corner of
$\lambda^{\beta(b_1)-k}$, and $Q(b_1)=k-1$, $Q(b_2)=0$, $P(b_1)=1$,
and $P(b_2)=n$, or
\item[(b)] there is a box $b_3<b_2$ with
$\mathrm{ct}(b_3)=\mathrm{b_2} \pm 1$ and so that
$Q(b_1)=Q(b_2)+k=Q(b_3)+k$ and $P(b_1)=P(b_3)+1=P(b_2)+2$.
\end{enumerate}  Assume that case (b) holds; (a) is similar.  

By hypothesis there is some integer $0<l<k$ such that
either $l=d_{\beta(b_1)}-d_{\beta(b_1)-l}+r \mathrm{ct}(b_1) c_0$ (in
which case $(P,Q) \notin \Gamma_c$) or so that there is a box $b_4 \in
\lambda^{\beta(b_1)-l}$ with 
$$l=d_{\beta(b_1)}-d_{\beta(b_4)}+r(\mathrm{ct}(b_1)-\mathrm{ct}(b_4)
\pm 1) c_0$$ in which case also
$$0<k-l=d_{\beta(b_4)}-d_{\beta(b_2)}+r(\mathrm{ct}(b_4)-\mathrm{ct}(b_2)
\pm 1) c_0$$ and $k-l=\beta(b_4)-\beta(b_2) \ \mathrm{mod} \ r$,
implying that 
\begin{enumerate}
\item[(1)] $Q(b_1) \leq Q(b_4)+l$ with equality implying
$P(b_1)>P(b_4)$ and
\item[(2)] $Q(b_4) \leq Q(b_2)+k-l$ with equality implying
$P(b_4)>P(b_2)$.  
\end{enumerate} Observe first that if $b_4 = b_3$ then the requirement $Q(b_1) \leq Q(b_4)+l<Q(b_4)+k$ precludes $(P,Q) \in \Gamma_c$.  Thus $b_4 \neq b_3$.  On the other hand, combining (1) and (2) above shows that $Q(b_1) \leq Q(b_2)+k$ with equality implying $P(b_1)>P(b_4)>P(b_2)$, and this contradicts $P(b_1)=P(b_3)+1=P(b_2)+2$.
\end{proof} 

The modules $\Delta_c(\lambda^\bullet)$ are graded by polynomial degree, and the modules $L_c(\lambda^\bullet)$ are graded quotients (thanks to the deformed Euler operator from \cite{DuOp}), on which $W$ acts preserving the degree.  Writing $L^d_c(\lambda^\bullet)$ for the degree $d$ piece of $L_c(\lambda^\bullet)$, we define the graded character
$$\mathrm{char}(L_c(\lambda^\bullet)^W,q)=\sum_{d=0}^\infty \mathrm{dim}_\CC(L^d_c(\lambda^\bullet)^W) q^d.$$  By using Theorem 2.1 of \cite{DuGr}, we obtain the following corollary of Theorem \ref{diag thm}, where we define a column-strict tableau on $\lambda^\bullet$ to be a filling of its boxes by non-negative integers in such a way that within each component $\lambda^i$, the entries are weakly increasing left to right, and strictly increasing top to bottom.

\begin{corollary} \label{invariants char}
If $c_0 \neq 0$ and $L_c(\lambda^\bullet)$ is diagonalizable, then the graded character of $L_c(\lambda^\bullet)^W$ is given by the formula
$$\mathrm{char}(L_c(\lambda^\bullet)^W,q)=\sum_{d=0}^\infty a_d q^d,$$ where $a_d$ is the number of column-strict tableaux $Q$ on $\lambda^\bullet$ such that $Q(b)=\beta(b) \ \mathrm{mod} \ r$ for all boxes $b$ of $\lambda$, the sum 
$$\sum_{b \in \lambda^\bullet} Q(b)=d,$$ and for each positive integer $l$ the following conditions hold:
\begin{enumerate}
\item[(a)] we have $Q(b)<l$ whenever $b \in \lambda^i$ and the equation $d_i-d_{i-l}+r \mathrm{ct}(b) c_0=l$ holds, and
\item[(b)]  we have $Q(b_1)-Q(b_2) \leq l$ whenever $b_1 \in \lambda^i$, $b_2 \in \lambda^{i-l}$, and one of the equations $$d_i-d_{i-l}+r(\mathrm{ct}(b_1)-\mathrm{ct}(b_2) \pm 1) c_0=l$$ holds.
\end{enumerate}
\end{corollary}  Writing $e=\frac{1}{|W|}\sum_{w \in W} w$ for the symmetrizing idempotent, if the parameter $c$ is not on one of the hyperplanes specified by Theorem 3.4 of \cite{DuGr}, then the functor $M \mapsto eM=M^W$ from $H_c$ modules to $eH_c e$-modules is an equivalence, and the module $L_c(\lambda^\bullet)$ is finite dimensional if and only if $L_c(\lambda^\bullet)^W$ is.

\section{Proof of Theorem \ref{unitary thm}}

\subsection{} 
By Theorem \ref{diag thm}  $L_c(\lambda^\bullet)$ is unitary exactly if it is diagonalizable and for all $(P,Q) \in \Gamma_c$ we have $\la f_{P,Q},f_{P,Q} \ra \geq 0$. Since every $f_{P,Q}$ for $(P,Q) \in \Gamma_c$ may be obtained from those with $Q=0$ by applying an invertible sequence of intertwining operators, the following lemma allows for inductive control over the signs of the norm of $f_{P,Q}$ for $(P,Q) \in \Gamma_c$. It is an immediate consequence of the proof of Theorem 6.1 of \cite{Gri2}, translated into $(P,Q)$ notation.
\begin{lemma}
\item[(a)] Suppose $(P,Q) \in \Gamma_c$. Then for $1 \leq i \leq n-1$ we have
$$\la \sigma_i f_{P,Q},\sigma_i f_{P,Q} \ra=\la f_{P,Q},f_{P,Q} \ra$$ if $Q(P^{-1}(i))-Q(P^{-1}(i+1)) \neq \beta(P^{-1}(i))-\beta(P^{-1}(i+1))$ and
$$\la \sigma_i f_{P,Q},\sigma_i f_{P,Q} \ra=\frac{(\delta-r c_0)(\delta+r c_0)}{\delta^2} \la f_{P,Q},f_{P,Q} \ra$$ with
$$\delta=Q(P^{-1}(i))-Q(P^{-1}(i+1))-(d_{\beta(P^{-1}(i))}-d_{\beta(P^{-1}(i+1))})-r(\mathrm{ct}(P^{-1}(i))-\mathrm{ct}(P^{-1}(i+1)))$$ else.
\item[(b)] For $(P,Q) \in \Gamma_c$ we have
$$\la \Phi f_{P,Q},\Phi f_{P,Q} \ra=\left(Q(P^{-1}(1))+1- (d_{\beta(P^{-1}(1))}-d_{\beta(P^{-1}(1))-Q(P^{-1}(1))-1})-r \mathrm{ct}(P^{-1}(1)) c_0 \right) \la f_{P,Q},f_{P,Q} \ra.$$
\end{lemma}

The lemma implies that for $c_0 \neq 0$, the module $L_c(\lambda^\bullet)$ is unitary if and only if the following two conditions hold: for all $(P,Q) \in \Gamma_c$,
\begin{equation} \label{1 res}
Q(P^{-1}(1))+1 \geq d_{\beta(P^{-1}(1))}-d_{\beta(P^{-1}(1))-Q(P^{-1}(1))-1}+r \mathrm{ct}(P^{-1}(1)) c_0
\end{equation} and if $(P,Q) \in \Gamma_c$ and $1 \leq i \leq n-1$ with $Q(P^{-1}(i))-Q(P^{-1}(i+1))=\beta(P^{-1}(i))-\beta(P^{-1}(i+1)) \ \mathrm{mod} \ r$ then setting $b_1=P^{-1}(i)$ and $b_2=P^{-1}(i+1)$
\begin{equation} \label{2 res}
\left(Q(b_1)-Q(b_2)-(d_{\beta(b_1)}-d_{\beta(b_2)})-r(\mathrm{ct}(b_1)-\mathrm{ct}(b_2))c_0 \right)^2 \geq (rc_0)^2.
\end{equation}  The last condition may be rephrased: the numbers
\begin{equation} \label{3 res}
Q(b_1)-Q(b_2)-(d_{\beta(b_1)}-d_{\beta(b_2)})-r(\mathrm{ct}(b_1)-\mathrm{ct}(b_2) \pm 1)c_0
\end{equation} have the same sign (are both weakly positive or both weakly negative). Assuming $c_0 >0$, the only way this can fail is if 
$$Q(b_1)-Q(b_2)-(d_{\beta(b_1)}-d_{\beta(b_2)})-r(\mathrm{ct}(b_1)-\mathrm{ct}(b_2) - 1)c_0>0$$ and 
$$Q(b_1)-Q(b_2)-(d_{\beta(b_1)}-d_{\beta(b_2)})-r(\mathrm{ct}(b_1)-\mathrm{ct}(b_2) + 1)c_0 < 0 .$$

\subsection{Construction of unitarity-preventing $(P,Q)$'s} The following lemma is the key step in the proof that the existence of appropriate blocking sequences is a necessary condition for unitarity.

\begin{lemma} \label{PQ construct}
\item[(a)] Let $b \in \lambda^i$ and $b' \in \lambda^j$ and suppose that $b \nleq b'$ and there is no blocking sequence $B$ for $(b,b')$ with $c \in L_B$. Then there exist $1 \leq a \leq n-1$ and  $(P,Q) \in \Gamma_c(\lambda^\bullet)$ with $P(b)=a+1$, $P(b')=a$, $Q(b)=m_{ij}$ and $Q(b')=0$.
\item[(b)] Let $b \in \lambda^i$ and $0 \leq j \leq r-1$ and suppose that there is no blocking sequence $B$ for $(b,j)$ with $c \in L_B$. Then there is some $(P,Q) \in \Gamma_c$ with $P(b)=1$ and $Q(b)=m_{ij}-1$. 
\end{lemma}

\begin{proof}

We describe a general procedure for producing an element $(P,Q) \in \Gamma_c$, with different initial steps for parts (a) and (b) of the lemma. For part (a): for each $b'' \geq b$ set $Q(b'')=m_{ij}$, and for each $b''' \leq b'$ set $Q(b''')=0$.  Define $P(b'')$ for all $b'' > b$ and all $b'''<b'$ in such a way that $P$ is decreasing on these posets, and furthermore so that the set of numbers thus defined is equal to the set $\{d,d+1,\dots,n\}$, where $n-d+1$ is the number of boxes at least $b$ or at most $b'$ (this last condition will force $P(b'')<d$ for the remaining boxes $b''$ of $\lambda^\bullet$). For part (b): set $Q(b)=m_{ij}-1$ and for each $b' > b$ set $Q(b')=m_{ij}$. Set $P(b)=1$ and define $P(b')$ on the set of boxes $b' > b$ in such a way that $P$ is decreasing on this poset and the set of numbers so used is of the form $\{d+1,\dots,n\}$, where $n-d$ is the number of boxes strictly larger than $b$. The remainder of the construction in both cases (a) and (b) of the lemma is now the same.

Assuming we have defined $Q$ and $P$ on all boxes in $\lambda^{i-1}, \lambda^{i-2},\dots,\lambda^{i-k+1}$, we define them on $\lambda^{i-k}$ by induction, choosing the minimal $b_1$ for which $Q$ and $P$ are not already defined.  Choose $P(b_1)$ maximal from among the unused numbers in $\{1,2,\dots,n\}$.  We choose $Q(b_1)$ minimal subject to the conditions:
\begin{enumerate}
\item[(a)] $Q(b_1) \geq 0$,
\item[(b)] $Q(b_1) \geq Q(b_2)$ for all $b_2 \leq b_1$, and
\item[(c)] for each box $b_2$ such that $Q(b_2)$ and $P(b_2)$ have already been defined, and for any positive integer $l$ with $l=\beta(b_2)-\beta(b_1) \ \mathrm{mod} \ r$ and 
$$l=d_{\beta(b_2)}-d_{\beta(b_1)}+r(\mathrm{ct}(b_2)-\mathrm{ct}(b_1) \pm 1)c_0,$$ we enforce $Q(b_2)-Q(b_1) \leq l$.
\end{enumerate} One now checks, using the absence of blocking sequences, that the pair $(P,Q)$ so defined belongs to $\Gamma_c$. 
\end{proof}

\subsection{Proof of Theorem \ref{unitary thm} in case $c_0=0$.} If $c_0=0$ then $\sigma_i=s_i$ so all $f_{P,Q}$ are well-defined, and by the argument of Theorem 6.1 of \cite{Gri2} we have
$$\la s_i f_{P,Q},s_i f_{P,Q} \ra=\la f_{P,Q},f_{P,Q} \ra$$ and
$$\la \Phi f_{P,Q},\Phi f_{P,Q} \ra=\left(Q(P^{-1}(1))+1-(d_{P^{-1}(1)}-d_{P^{-1}(1)-Q(P^{-1}(1))-1}) \right) \la f_{P,Q},f_{P,Q} \ra.$$ It follows that if $d_i-d_j \leq m_{ij}$ for all pairs $i,j$ then the contravariant form is positive semi-definite on $\Delta_c(\lambda^\bullet)$ and hence $L_c(\lambda^\bullet)$ is unitary, and that if $d_i-d_j > m_{ij}$ for some $i$ with $\lambda^i \neq \emptyset$ then there is some $f_{P,Q}$ with negative square norm and hence $L_c(\lambda^\bullet)$ is not unitary.

\subsection{} From now on we assume $c_0 > 0$. We first prove that if one of the conditions (a), (b) or (c) in the statement of Theorem \ref{unitary thm} fails to hold then $L_c(\lambda^\bullet)$ is not unitary. If (a) fails, then $L_c(\lambda^\bullet)$ is not diagonalizable, and hence it is not unitary. If (b) fails then there is a box $b \in \lambda^i$ and an integer $j$ for which $$m_{ij} < d_i-d_j+ rÊ\mathrm{ct}(b) c_0$$ and so that there is no blocking sequence $B$ for $(b,j)$ with $c \in L_B$. Then by Lemma \ref{PQ construct} there is some $(P,Q) \in \Gamma_c$ with $P(b)=1$ and $Q(b)=m_{ij}-1$, and using \eqref{1 res} shows that $L_c(\lambda^\bullet)$ is not unitary. Finally, suppose that there are boxes $b \in \lambda^i$ and $b' \in \lambda^j$ with 
$$m_{ij}< d_i-d_j+r (\mathrm{ct}(b)-\mathrm{ct}(b')+1)c_0$$ but so that there is no blocking sequence $B$ for $(b,b')$ with $c \in L_B$. Since a blocking sequence for $(b,j)$ would be one for $(b,b')$, there is no blocking sequence $B$ for $(b,j)$ with $c \in L_B$ and hence by the previous part of the proof we may assume that
$$d_i-d_j+r \mathrm{ct}(b) c_0 \leq m_{ij}.$$ Let $k \in \ZZ$ be the largest integer so that 
$$m_{ij} < d_i-d_j+r (\mathrm{ct}(b)-k+1)c_0,$$ so that we have $\mathrm{ct}(b') \leq k \leq 0$. It follows that there is some $b'' \leq b'$ with $\mathrm{ct}(b'')=k$. Since a blocking sequence $B$ for $(b,b'')$ is one for $(b,b')$, there is no blocking sequence $B$ for $(b,b'')$ with $c \in L_B$. In particular $d_i-d_j+r(\mathrm{ct}(b)-\mathrm{ct}(b'')-1) c_0 \neq m_{ij}$ (since $(b,b'')$ is a blocking sequence for $(b,b'')$) and it follows that
$$d_i-d_j+r(\mathrm{ct}(b)-\mathrm{ct}(b'')-1) c_0 < m_{ij} < d_i-d_j+r(\mathrm{ct}(b)-\mathrm{ct}(b'')+1) c_0.$$ Moreover Lemma \ref{PQ construct} implies that there is some $(P,Q)Ê\in \Gamma_c$ with $P(b)=i$, $P(b'')=i+1$, $Q(b)=m_{ij}$, and $Q(b'')=0$. Now \eqref{3 res} shows $L_c(\lambda^\bullet)$ is not unitary. 

\subsection{}  For the converse, we must prove that if (a), (b) and (c) in the statement of Theorem \ref{unitary thm} hold then $L_c(\lambda^\bullet)$ is unitary. We assume that $(P,Q) \in \Gamma$ and $b_1=b$, $b_2=b'$ are such that the numbers in \eqref{3 res} have different signs, and we will show $(P,Q) \notin \Gamma_c$. First, we may assume that $b$ is not at most $b'$, by interchanging them if necessary. Then the inequality
$$m_{ij}<d_{\beta(b)}-d_{\beta(b')}+r(\mathrm{ct}(b)-\mathrm{ct}(b')+1)c_0,$$ implies that there is a blocking sequence $B$ for $(b,b')$ with $c \in L_B$. Now Lemma~\ref{blocking1} below implies that $(P,Q)$ satisfying \eqref{3 res} cannot be in $\Gamma_c$.  Likewise if \eqref{1 res} fails for $(P,Q) \in \Gamma_c$ then with $b=P^{-1}(1)$, $i=\beta(b)$, and $j=\beta(b)-Q(b)-1$ we have
$$m_{ij} <d_i-d_j+r \mathrm{ct}(b) c_0,$$ and hence there exists a blocking sequence $B$ for $(b,j)$ with $c \in L_B$, and Lemma~\ref{blocking2} below implies that $(P,Q) \notin \Gamma_c$. This completes the proof of Theorem \ref{unitary thm}.

\begin{lemma}\label{blocking2} If $b \in \lambda^i$ and $0 \leq j \leq r-1$ is a pair for which a blocking sequence $B$ with $c \in L_B$ exists, and there is $(P,Q) \in \Gamma_c$ with $Q(b) \geq m_{ij}-1$ and $P(b)=1$ then we must have
$$d_{\beta(b)}-d_j+r\mathrm{ct}(b)c_0=m_{ij}.$$
\end{lemma} 
\begin{proof}  If $B=(b_1,\dots,b_{2q+1},\ell)$ is a blocking sequence for $(b,j)$ with $c \in L_B$ then we obtain
$$Q(b) \leq Q(b_{2q+1})+ \sum_{k=1}^ {q} m_{\beta(b_{2k-1}),\beta(b_{2k})} \leq m_{\beta(b_{2q+1}),\ell}-1+\sum_{k=1}^ {q} m_{\beta(b_{2k-1}),\beta(b_{2k})} \leq m_{ij}-1,$$ with equality implying $1=P(b) \geq P(b_{2q+1})+d-1$, where $d$ is the number of distinct boxes appearing in the sequence $b,b_1,\dots,b_{2q+1}$.  It follows that $d=1$, that $q=0$, that $b=b_1$, and that 
$$d_{\beta(b_1)}-d_l+r \mathrm{ct}(b_1) c_0=m_{il}.$$ But now we have $m_{ij} \leq m_{il}$ since $(P,Q) \in \Gamma_c$, and the opposite inequality follows from the definition of blocking sequence.  Thus $m_{ij}=m_{il}$, whence $l=j$ and we are done.
\end{proof}

\begin{lemma} \label{blocking1} If $b \in \lambda^i$, $b' \in \lambda^j$, a blocking sequence for $(b,b')$ exists, and $(P,Q) \in \Gamma_c$ with $Q(b)-Q(b') \geq m_{ij}$ and $P(b)=P(b')+1$, then $Q(b)=Q(b')+m_{ij}$ and
$$m_{ij}=d_i-d_j+r (\mathrm{ct}(b)-\mathrm{ct}(b') \pm 1) c_0.$$
\end{lemma} 
\begin{proof}  If there is a blocking sequence $B$ for $(b,j)$ then the proof of Lemma \ref{blocking2} gives $Q(b) \leq m_{ij}-1$, contradiction. So any blocking sequence $B$ for $(b,b')$ with $c \in L_B$ must be of the form $B=(b_1,\dots,b_{2q})$.  Since $(P,Q) \in \Gamma$, for $1 \leq k \leq q-1$ we have $Q(b_{2k}) \leq Q(b_{2k+1})$ with equality implying that $P(b_{2k})>P(b_{2k+1})$ unless $b_{2k}=b_{2k+1}$.  Since $(P,Q) \in \Gamma_c$, for $1 \leq k \leq q$ we have $Q(b_{2k-1}) \leq Q(b_{2k})+m_{\beta(b_{2k-1}),\beta(b_{2k})}$ with equality implying $P(b_{2k-1})>P(b_{2k})$.  Combining these inequalities, we have 
$$Q(b_0) \leq Q(b')+ \sum_{k=1}^q m_{\beta(b_{2k-1}),\beta(b_{2k})} \leq Q(b')+m_{ij},$$ with equality implying $P(b) \geq P(b')+d-1$ where $d$ is the number of distinct boxes that appear in the sequence $b,b_1,\dots,b_{2q},b'$.  Since $P(b)=P(b')+1$ we must have $d=2$, whence $b_1=b$, $b_2=b'$, and 
$$d_{\beta(b)}-d_{\beta(b')}+r(\mathrm{ct}(b)-\mathrm{ct}(b') \pm 1)c_0=m_{ij}$$ so we are done.
\end{proof}

\section{Explicit results for classical Weyl groups} \label{classical type}

\subsection{The symmetric group} Here we assume $r=1$. The parameter is a single number $c=c_0$, which we may assume is non-negative $c \geq 0$. The next corollary recovers the result of \cite{EtSt} classifying the unitary representations of the type A rational Cherednik algebra, and the classification due to Suzuki (\cite{Suz}; see also \cite{SuVa}) of the diagonalizable irreducible representations. The following diagram should help the reader to visualize the notation introduced in the corollary.

$$
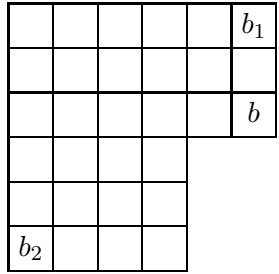
\begin{ytableau}
 \ & & & & & b_1 \\
 & & & & & \\
 & & & & & b \\
 & & &  \\
 & & & \\
 b_2 & & & 
\end{ytableau}$$

\begin{corollary} \label{type a corollary}
Suppose $c \geq 0$ and $\lambda$ is a partition of $n$. If $\lambda=(1^n)$, then $L_c(\lambda)$ is unitary for all $c \geq 0$. Otherwise, let $b_1$ be the box of $\lambda$ of maximum content, let $b_2$ be the box of $\lambda$ of minimum content, and let $b$ be the removable box of $\lambda$ of maximum content. Then $L_c(\lambda)$ is diagonalizable unless $c=k/m$ for relatively prime positive integers $k$ and $m$ with $m \leqÊ\mathrm{ct}(b)-\mathrm{ct}(b_2)$, and $L_c(\lambda)$ is unitary if and only if either $$c \leq \frac{1}{\mathrm{ct}(b_1)-\mathrm{ct}(b_2)+1}$$ or $c=1/m$ for an integer $m$ with $$m \geq \mathrm{ct}(b)-\mathrm{ct}(b_2)+1.$$
\end{corollary}  
\begin{proof}
We define $b(\lambda)=\mathrm{ct}(b)-\mathrm{ct}(b_2)+1$. We have $k_c(b)=\infty$ unless $c=k/m$ for relatively prime positive integers $k$ and $m$ with $1 \leq m \leq b(\lambda)-1$, in which case $k_c(b)=k=l_c(b)$. Moreover, if $k_c(b)=\infty$ then the same is true for all removable boxes of $\lambda$. It now follows from Theorem \ref{diag thm} that for $c \geq 0$ the module $L_c(\lambda)$ is $\ttt$-diagonalizable if and only if $c$ is not a rational number of the form $c=k/m$ with $1 \leq m \leq b(\lambda)-1$. This recovers the classification from \cite{Suz} (see also \cite{SuVa}).   

We now compute the set of $c \geq 0$ for which $L_c(\lambda)$ is unitary. If $\lambda=(1^n)$ then $k_c(b)=\infty$ for the only removable box, so it is diagonalizable for all $c \geq 0$, and since $\mathrm{ct}(b) c \leq 0$ for all boxes $b$ of $\lambda$ and $b_1 \leq b_2$ for all pairs of boxes with $(\mathrm{ct}(b_1)-\mathrm{ct}(b_2)+1) c >0$, Theorem \ref{unitary thm} shows that $L_c((1^n))$ is unitary for all $c \geq 0$. So from now on we assume $\lambda \neq (1^n)$. 

 Suppose first that $L_c(\lambda)$ is unitary. If $c > 1/(\mathrm{ct}(b_1)-\mathrm{ct}(b_2)+1)$, then since $\lambda \neq (1^n)$ we do not have $b_1 \leq b_2$, and hence by Theorem \ref{unitary thm} there must exist a blocking sequence $B$ for $(b_1,b_2)$ with $c \in L_B$. The definition of $L_B$ and our assumption that $c \geq 0$ now shows that $c=1/m$ for some positive integer $m$. Since $L_c(\lambda)$ is diagonalizable, the previous paragraph implies $m \geq b(\lambda)$. Thus we have either $c \leq 1/(\mathrm{ct}(b_1)-\mathrm{ct}(b_2)+1)$ or $c=1/m$ for an integer $m$ with $m \geq b(\lambda)$. 

Conversely, if $0 \leq c \leq 1/(\mathrm{ct}(b_1)-\mathrm{ct}(b_2)+1)$ then we have already seen that $L_c(\lambda)$ is diagonalizable. If $b \in \lambda$ then 
$$\mathrm{ct}(b) c \leq (\mathrm{ct}(b_1)-\mathrm{ct}(b_2)+1)c < 1,$$  and likewise we have $$(\mathrm{ct}(b)-\mathrm{ct}(b')+1)c \leq (\mathrm{ct}(b_1)-\mathrm{ct}(b_2)+1) c \leq 1$$ for all pairs $b,b'$ of boxes of $\lambda$, so by Theorem \ref{unitary thm} $L_c(\lambda)$ is unitary.

Finally suppose $c=1/m$ for some integer $m \geq b(\lambda)$. We have seen above that in this case $L_c(\lambda)$ is diagonalizable. Observe first that we always have $b(\lambda) \geq \mathrm{ct}(b_1)$. Thus for all $b' \in \lambda$, we have
$$\mathrm{ct}(b') c \leq \mathrm{ct}(b_1) \cdot 1/ b(\lambda) \leq 1.$$ Hence (c) of Theorem \ref{unitary thm} holds. If $b',b'' \in \lambda$ are boxes with 
$$(\mathrm{ct}(b')-\mathrm{ct}(b'')+1)c > 1$$ then 
$$\mathrm{ct}(b')-\mathrm{ct}(b'')+1 > m.$$ We show that there is a blocking sequence $B$ for $(b',b'')$ with $c \in L_B$. If $b''' \leq b''$ then a blocking sequence for $(b',b''')$ is one for $(b',b'')$ and if $b''' \geq b'$ then a blocking sequence for $(b''',b'')$ is one for $(b',b'')$. Furthermore, moving $b'$ to the right increases its content and moving $b''$ to the left decreases its content. So we may assume that there is no box to the right of $b'$ in $\lambda$ and there is no box to the left of $b''$ in $\lambda$. The inequality
$$\mathrm{ct}(b')-\mathrm{ct}(b'')+1 >m \geq b(\lambda) \quadÊ\implies \mathrm{ct}(b') > b(\lambda)+\mathrm{ct}(b'') -1 \geq \mathrm{ct}(b)  $$ then implies that $b_1 \leq b' < b$, and hence there is a box $b'''$ with $b' < b''' \leq b$ and
$$\mathrm{ct}(b''')-\mathrm{ct}(b'')+1=m,$$ so that $B=(b''',b'')$ is a blocking sequence for $(b',b'')$ with $c \in L_B$. 
\end{proof}

\subsection{Type B} For $r=2$ (essentially, for the Weyl groups of type $B$/$C$ and $D$) our results can also be made completely explicit. The Weyl group of type $B_n$ is $G(2,1,n)$, which contains two conjugacy classes of reflections: those conjugate to the transposition $(12)$, and those conjugate to the transformation $(x_1,x_2,\dots,x_n) \mapsto (-x_1,x_2,\dots,x_n)$ changing the sign of the first coordinate and leaving the remaining coordinates fixed. We will write $c$ and $d$ for the parameters $c_r$ attached to these conjugacy classes; in terms of $c_0$, $d_0$, and $d_1$ we then have
$$c=c_0 \quad \text{and} \quad d=d_0=-d_1.$$

First we will assume $\lambda^\bullet=(\lambda,\emptyset)$ where $\lambda$ is a partition of $n$.  For the next two corollaries, we need some notation for certain boxes of $\lambda$, and a diagram illustrating the notation. We suppose that $\lambda=(6,6,6,4,4,4)$ with Young diagram and certain marked boxes as illustrated here:
$$\begin{ytableau}
 \ & & & & & b_1 \\
 & & & & & \\
 & & & & & b_2 \\
 & & & b_3 \\
 & & & \\
 b_5 & & & b_4
\end{ytableau} $$ The significance of these particular boxes is as follows: $b_1$ is the box of $\lambda$ of largest content; $b_2$ is the removable box of largest content; $b_4$ is the removable box of second-largest content; $b_3$ is the highest box in the right rim of $\lambda$ and directly above $b_4$; finally, $b_5$ is the box of smallest content. If $\lambda$ is a rectangle then the boxes $b_3$ and $b_4$ do not exist, and we will not make use of them in this case. It may happen that some of these boxes are the same. The analog of the theorem of Suzuki \cite{Suz} and Cherednik \cite{Che1} in this case is contained in the following corollaries.

\begin{corollary} \label{B diag 1}
Suppose $\lambda$ is a rectangle or $b_4=b_5$, and let $\lambda^\bullet=(\lambda,\emptyset)$. If $\lambda=(1^n)$, then $L_c(\lambda^\bullet)$ is diagonalizable for all $c \geq 0$. Otherwise, $L_c(\lambda^\bullet)$ is diagonalizable if and only if
\begin{enumerate}
\item[(a)] $c$ is not a rational number of denominator at most $\mathrm{ct}(b_2)-\mathrm{ct}(b_5)$, or
\item[(b)] an equation of the form
$$d+\mathrm{ct}(b_2) c=m/2$$ holds for some positive odd integer $m$, and $c$ is not of the form $c=k/\ell$ for positive coprime integers $k$ and $\ell$ with $2k \leq m$. 
\end{enumerate}
\end{corollary}

\begin{corollary} \label{B diag 2}
Suppose $\lambda$ is not a rectangle and $b_4 \neq b_5$. With notation as above, and assuming $c \geq 0$ and $\lambda^\bullet=(\lambda,\emptyset)$, then $L_c(\lambda^\bullet)$ is diagonalizable if and only if
\begin{enumerate}
\item[(a)] $c$ is not a rational number of denominator at most $\mathrm{ct}(b_2)-\mathrm{ct}(b_5)$, or
\item[(b)] $c=k/l$ for coprime positive integers $k$ and $\ell$ such that $\mathrm{ct}(b_4)-\mathrm{ct}(b_5)+1 \leq \ell \leq \mathrm{ct}(b_2)-\mathrm{ct}(b_5)$ and an equation of the form
$$d+ \mathrm{ct}( b_2) c=m/2$$ holds for some positive odd integer $m<2k$.  
\end{enumerate}
\end{corollary} One proves these corollaries by the same technique as the type A case. Using these corollaries we will deduce the next result on unitarity.

\begin{corollary} \label{B corollary 1}
If $\lambda=(1^n)$ and $c \geq 0$, then $L_c(\lambda^\bullet)$ is unitary if and only if either $d \leq 1/2$ or $d+\ell c=1/2$ for some integer $\ell$ with $-(n-1) \leq \ell \leq -1$. Assuming $c \geq 0$ and $\lambda \neq (1^n)$, the representation $L_c(\lambda,\emptyset)$ is unitary if and only if $(c,d)$ belongs to at least one of the following sets:
\begin{enumerate}
\item[(a)]  The set of parameters $(c,d)$ satisfying the inequalities $$c \leq \frac{1}{\mathrm{ct}(b_1)-\mathrm{ct}(b_5)+1} \quadÊ\text{and} \quad d+\mathrm{ct}(b_1)c \leq \frac{1}{2},$$
\item[(b)] for each positive integer $\ell$ with $\mathrm{ct}(b_2)-\mathrm{ct}(b_5)+1 \leq \ell \leq \mathrm{ct}(b_1)-\mathrm{ct}(b_5)+1$, the ray consisting of points $(c,d)$ such that $c=1/\ell$ and $$d+ \mathrm{ct}(b_1) c \leq \frac{1}{2},$$
\item[(c)] for each positive integer $\ell$ with $\mathrm{ct}(b_2) < \ell \leq \mathrm{ct}(b_1)$, the segment consisting of points $(c,d)$ with $$d+ \ell c=\frac{1}{2} \ \text{and} \ c \leq \frac{1}{\ell-\mathrm{ct}(b_5)},$$
\item[(d)] the points $(c,d)$ such that $d+\mathrm{ct}(b_2) c=\frac{1}{2}$, and, if (i) $\lambda$ is not a rectangle and (ii) $b_4 \neq b_5$, such that either $$c \leq \frac{1}{\mathrm{ct}(b_3)-\mathrm{ct}(b_5)+1} \quad \text{or} \quad c=\frac{1}{\ell-\mathrm{ct}(b_5)+1}$$ for an integer $\ell$ satisfying $\mathrm{ct}(b_4) \leq \ell < \mathrm{ct}(b_3)$.
\item[(e)] For each pair $(\ell,m)$ of integers with $\mathrm{ct}(b_2)+1 \leq \ell \leq \mathrm{ct}(b_1)-1$  and $\mathrm{ct}(b_2)-\mathrm{ct}(b_5)+1 \leq m \leq \ell-\mathrm{ct}(b_5)$ the point $P_{\ell,m}$ satisfying 
$$d+\ell c=1/2 \quad \text{and} \quad c=1/m.$$
\end{enumerate}
\end{corollary}
\begin{proof}
Theorem \ref{unitary thm} implies that if $c=0$ then $L_c(\lambda^\bullet)$ is unitary if and only if $d \leq 1/2$.  Suppose first that $(c,d)$ is such that $L_c(\lambda^\bullet)$ is unitary. If $\lambda=(1^n)$ and $d>1/2$, then by Theorem \ref{unitary thm} there must exist a blocking sequence $B$ for $(b_1,1)$ with $c \in L_B$. Thus an equation of the form $d+\ell c=1/2$ holds for some $-(n-1) \leq \ell \leq -1$. So we may suppose $\lambda \neq (1^n)$ and $c>0$. 

We will prove that $(c,d)$ is of one of the types (a)-(e) in the statement of the theorem above. We may assume that at least one of the inequalities in (a) fails. 

Case 1: if $d+\mathrm{ct}(b_1) c \leq 1/2$ then we must have $$c >  \frac{1}{\mathrm{ct}(b_1)-\mathrm{ct}(b_5)+1}.$$  Suppose first that $d+\mathrm{ct}(b_1)c < 1/2$. It follows that $d+\mathrm{ct}(b)c < 1/2$ for all boxes $b \in \lambda$, and hence $c \notin L_B$ for all blocking sequences of the form $(b,1)$. Since $\lambda  \neq (1^n)$ we do not have $b_1 \leq b_5$ and by Theorem \ref{unitary thm} there must exist a blocking sequence $B=(b,b')$ for $(b_1,b_2)$ with $c \in L_B$. This implies $c=1/\ell$ for some positive integer $\ell$. By Corollary \ref{B diag 2} we must have $\ell \geq \mathrm{ct}(b_2)-\mathrm{ct}(b_5)+1$ and we are in case (b). 

Suppose next that $d+\mathrm{ct}(b_1)c=1/2$ (thus $(b_1,1)$ is a blocking sequence for $(b_1,b_5)$ with $c \in L_B$ so we cannot quite conclude as we just did). Either $b_1=b_2$ or there is a box $b >b_1$ with $\mathrm{ct}(b)=\mathrm{ct}(b_1)-1$. We will treat these two subcases next.

Subcase (i): $b_1=b_2$. If $\lambda$ is a rectangle (necessarily $\lambda=(n)$), if $b_4=b_5$, or if $$c \leq \frac{1}{\mathrm{ct}(b_3)-\mathrm{ct}(b_5)+1}$$ then we are in case (d). So we suppose $\lambda$ is not a rectangle, $b_4 \neq b_5$, and we have
$$c > \frac{1}{\mathrm{ct}(b_3)-\mathrm{ct}(b_5)+1}.$$ Since $b_4 \neq b_5$ we do not have $b_3 \leq b_5$, and (b) of Theorem \ref{unitary thm} implies that there is a blocking sequence $B$ for $(b_3,b_5)$ with $c \in L_B$. Since $d+\mathrm{ct}(b_1)c=1/2$ there is no blocking sequence $B=(b,1)$ for $(b_3,b_5)$ with $c \in L_B$. Thus there is a blocking sequence $(b,b')$ for $(b_3,b_5)$ with $c \in L_B$, implying $c=1/\ell$ for some positive integer $\ell$. By Corollary \ref{B diag 2} we have $\ell \geq \mathrm{ct}(b_4)-\mathrm{ct}(b_5)+1$, implying that we are in case (d). 
 
Subcase (ii):  If there is a box $b >b_1$ with $\mathrm{ct}(b)=\mathrm{ct}(b_1)-1$ then
$$c > \frac{1}{\mathrm{ct}(b)-\mathrm{ct}(b_5)+1},$$ and hence a blocking sequence $B$ for $(b,b_5)$ must exist with $c \in L_B$. We cannot have $B=(b',1)$ since $c>0$ implies that only one equation of the form $d+\ell c=1/2$ can hold and we are assuming $d+\mathrm{ct}(b_1) c=1/2$ already. So $B=(b',b'')$ and $c=1/\ell$ for some positive integer $\ell$. By Corollaries \ref{B diag 1} and \ref{B diag 2} we have $\ell \geq \mathrm{ct}(b_2)-\mathrm{ct}(b_5)+1$, and we are in case (b).

Case 2: if $d+\mathrm{ct}(b_1) c > 1/2$ then there must be a blocking sequence $B$ for $(b_1,1)$ with $c \in L_B$. This blocking sequence must be of the form $(b,1)$ for some $b > b_1$. Thus an equation $d+\ell c=1/2$ holds with $\mathrm{ct}(b_2) \leq \ell=\mathrm{ct}(b) < \mathrm{ct}(b_1)$. Suppose first that $\ell > \mathrm{ct}(b_2)$. If $c \leq \frac{1}{\ell-\mathrm{ct}(b_5)}$ then we are in case (c) above. Thus we may assume
$$c > \frac{1}{\ell-\mathrm{ct}(b_5)}=\frac{1}{\mathrm{ct}(b')-\mathrm{ct}(b_5)+1}$$ where $b'$ is the box directly below $b$ (there is such a box since $\ell > \mathrm{ct}(b_2)$). Hence there must exist a blocking sequence $B'$ for $(b',b_5)$ with $c \in L_{B'}$. Since $c>0$ the only equation of the form $d+\ell' c=1/2$ that holds is $d+\ell c=1/2$, and hence this blocking sequence is necessarily of the form $B'=(b'',b''')$ for some $b'' \geq b'$ and $b''' \leq b_5$. Thus $c=1/m$ for some integer $m$ with
$$m=\mathrm{ct}(b'')-\mathrm{ct}(b''')+1 \leq \mathrm{ct}(b')-\mathrm{ct}(b_5)+1 \leq \ell-\mathrm{ct}(b_5).$$ Since $L_c(\lambda^\bullet)$ is diagonalizable, we have $m \geq \mathrm{ct}(b_2)-\mathrm{ct}(b_5)+1$ by Corollaries \ref{B diag 1} and \ref{B diag 2}. Thus we are in case (e).

Finally we suppose that $\ell=\mathrm{ct}(b_2)$, or in other words $b=b_2$. We now repeat the argument of Subcase (i) of Case 1 above to conclude that we are in case (d). This completes the proof of necessity of at least one of the conditions (a)-(e). 

For the converse, we first suppose $\lambda=(1^n)$. By Corollary \ref{B diag 1} $L_c(\lambda^\bullet)$ is diagonalizable for all $c \geq 0$. The condition (b) in Theorem \ref{unitary thm} always holds so we need only check condition (c) there to verify unitarity. If $d \leq 1/2$ then we have $d+\mathrm{ct}(b) c \leq d \leq 1/2$ for all boxes $b$ and hence (c) holds. On the other hand, if $d+\mathrm{ct}(b) c=1/2$ for some box $b \in \lambda$, then $B=(b,1)$ is a blocking sequence for all $(b',1)$ and $(b',2)$ with $b' \leq b$. Moreover for any $b'$ with $b' \nleq b$ we have
$$d+\mathrm{ct}(b') c<d+\mathrm{ct}(b) c=1/2,$$ so condition (c) holds.

Now we suppose $\lambda \neq (1^n)$ and treat cases (a)-(e) in the statement one at a time. First we verify diagonalizability in each case. In cases (a), (b), (c), and (e) we have 
$$c \leq 1/(\mathrm{ct}(b_2)-\mathrm{ct}(b_5)+1)$$ and hence $L_c(\lambda^\bullet)$ is diagonalizable by Corollaries \ref{B diag 1} and \ref{B diag 2}. In case (d), if $\lambda$ is a rectangle then Corollary \ref{B diag 1} shows that $L_c(\lambda^\bullet)$ is diagonalizable, and if $\lambda$ is not a rectangle then Corollary \ref{B diag 2} shows that $L_c(\lambda^\bullet)$ is diagonalizable. 

It remains to verify that conditions (b) and (c) from Theorem \ref{unitary thm} hold provided that one of (a)-(e) of the present corollary does. In case (a) of this corollary conditions (b) and (c) of Theorem \ref{unitary thm} automatically hold. In case (b) of this corollary the inequality $d+\mathrm{ct}(b_1) c \leq 1/2$ ensures that condition (c) of Theorem \ref{unitary thm} holds, and arguing as in the last part of the proof of Corollary \ref{type a corollary} shows that condition (b) of Theorem \ref{unitary thm} holds. 

Suppose now that we are in case (c) of this corollary, with $d+\ell c=1/2$ and 
$$c \leq \frac{1}{\ell-\mathrm{ct}(b_5)}$$ for some $\mathrm{ct}(b_2) < \ell \leq \mathrm{ct}(b_1)$. Let $b \in \lambda$ be the box with $\mathrm{ct}(b)=\ell$ and $b_1 \leq b \leq b_2$. Thus $B=(b,1)$ is a blocking sequence with $c \in L_B$ for all $(b',1)$ and all $(b',b'')$ with $b' \leq b$. If
$$d+\mathrm{ct}(b') c > 1/2$$ then we must have $\ell < \mathrm{ct}(b')$. We prove that $B=(b,1)$ is a blocking sequence for $(b',1)$. We may assume that there is no box to the right of $b'$, and hence $b' \leq b$. A similar argument shows that if $\mathrm{ct}(b') c >1$ then $b' \leq b$. Hence (c) of Theorem \ref{unitary thm} holds. If  $b',b''$ are boxes with
$$(\mathrm{ct}(b')-\mathrm{ct}(b'')+1) c > 1$$ then
$$1<(\mathrm{ct}(b')-\mathrm{ct}(b'')+1) c \leq \frac{\mathrm{ct}(b')-\mathrm{ct}(b'')+1}{\ell-\mathrm{ct}(b_5)} \leq  \frac{\mathrm{ct}(b')-\mathrm{ct}(b_5)+1}{\ell-\mathrm{ct}(b_5)}$$ and hence $\mathrm{ct}(b') \geq \ell=\mathrm{ct}(b)$. We now conclude as before that $(b,1)$ is a blocking sequence for $(b',b'')$, and hence (b) of Theorem \ref{unitary thm} holds.

We now assume we are in case (d) of this corollary. Since $d+\mathrm{ct}(b_2)  c=1/2$ then $B=(b_2,1)$ is a blocking sequence with $c \in L_B$ for all $(b',1)$, $(b',2)$, and $(b',b'')$ with $b' \leq b_2$. If $b' \nleq b_2$ then we have $\mathrm{ct}(b') \leq \mathrm{ct}(b_2)$ and hence 
$$d+\mathrm{ct}(b') c \leq 1/2.$$ If $\lambda$ is a rectangle there are no $b' \nleq b_2$ and if $b_4=b_5$ then any $b' \nleq b_2$ has $\mathrm{ct}(b') \leq 0$ so that $\mathrm{ct}(b') c \leq 0 \leq 1$. It follows that (c) of Theorem \ref{unitary thm} holds in these two cases. Otherwise  either $$c \leq \frac{1}{\mathrm{ct}(b_3)-\mathrm{ct}(b_5)+1} \quad \text{or} \quad c=\frac{1}{\ell-\mathrm{ct}(b_5)+1}$$ for an integer $\ell$ satisfying $\mathrm{ct}(b_4) \leq \ell < \mathrm{ct}(b_3)$. Note that for all $b' \nleq b_2$ we have $\mathrm{ct}(b') \leq \mathrm{ct}(b_4)-\mathrm{ct}(b_5)+1$ (c.f. the proof of Corollary \ref{type a corollary}), implying $\mathrm{ct}(b') c \leq 1$ for all $b' \nleq b_2$. This establishes (c) of Theorem \ref{unitary thm} in this case. To establish (b) we may assume $(\mathrm{ct}(b')-\mathrm{ct}(b'')+1) c > 1$ for some $b' \nleq b_2$ and $b'' \in \lambda$. If $$c \leq \frac{1}{\mathrm{ct}(b_3)-\mathrm{ct}(b_5)+1}$$ this is impossible, so we may assume 
$$c=\frac{1}{\ell-\mathrm{ct}(b_5)+1}$$ for an integer $\ell$ satisfying $\mathrm{ct}(b_4) \leq \ell < \mathrm{ct}(b_3)$. But now the same argument as in the proof of Corollary \ref{type a corollary} establishes (b) of Theorem \ref{unitary thm}.

Finally we assume we are in case (e) of the present corollary. Let $b$ be the box of the right rim of $\lambda$ with $\ell=\mathrm{ct}(b)$. The equality $d+\ell c=1/2$ ensures that $B=(b,1)$ is a blocking sequence with $c \in L_B$ for all $(b',1), (b',2)$, and $(b',b'')$ with $b' \leq b$. Moreover if $b' \nleq b$ then we have $\mathrm{ct}(b') < \mathrm{ct}(b)$ and hence
$$d+\mathrm{ct}(b') c < d+\mathrm{ct}(b)c=1/2.$$ 
Now the equality $c=1/m$ for some integer $m$ with $\mathrm{ct}(b_2)-\mathrm{ct}(b_5)+1 \leq m \leq \ell-\mathrm{ct}(b_5)$ implies as in the proof of Corollary \ref{type a corollary} that $\mathrm{ct}(b') c \leq 1$ for all $b' \in \lambda$, and that if $(\mathrm{ct}(b')-\mathrm{ct}(b'')+1)c > 1$ then there is a blocking sequence $B$ for $(b',b'')$ with $c \in L_B$. We have completed the proof.

\end{proof}

The unitary spectrum for $\lambda^\bullet=((6,6,6,4,4,4),\emptyset)$ is drawn next. Again, the shaded area (which should be understood to extend infinitely downwards) is the region in which the standard module itself is unitary.
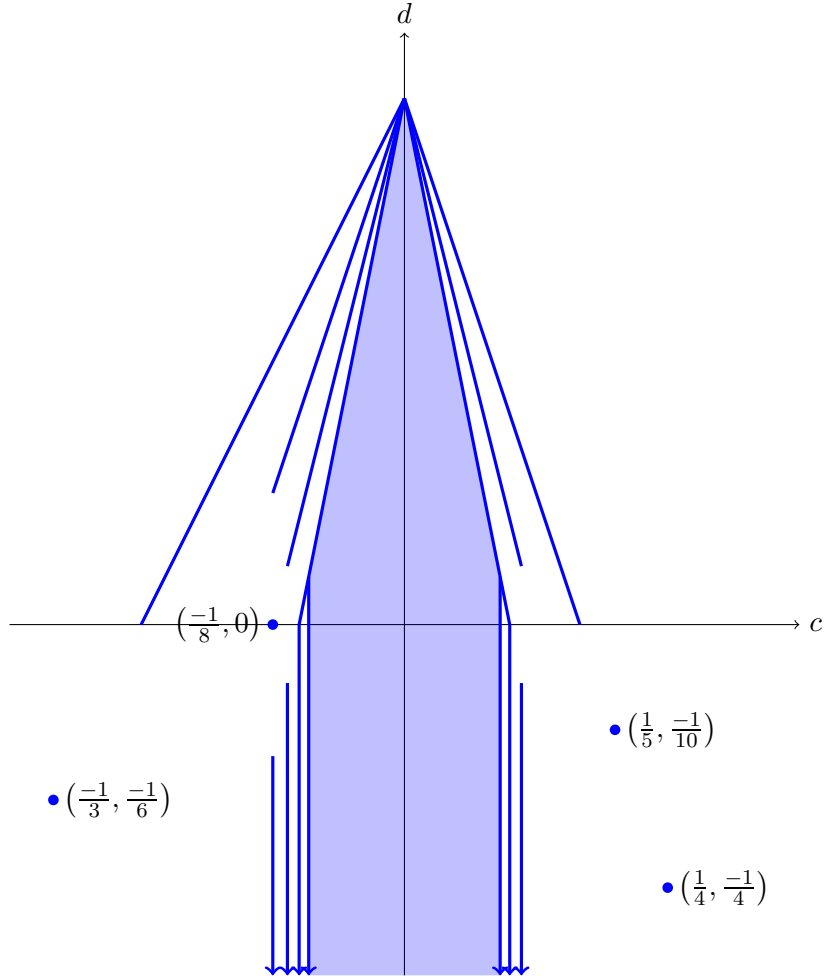
\begin{center}
\begin{tikzpicture}[scale=14]
\tikzstyle{axes}=[]
\tikzstyle{wall}=[thick]
\tikzstyle{relevant wall}=[very thick]
\tikzstyle{dot}=[fill]
\begin{scope}[style=axes]
\draw[->] (-3/8,0) -- (3/8,0) node[right] {$c$} coordinate(c1 axis);
\draw[->] (0,-1/3) -- (0,9/16) node[above] {$d$} coordinate(c2 axis);
\end{scope}

\begin{scope}[very thick,blue,auto=left]
\draw (0, 1/2) -- (1/10,0);
\draw (0, 1/2) -- (-1/10,0);
\draw (0,1/2) -- (1/9,1/18);
\draw (0,1/2) -- (1/6,0);
\draw (0,1/2) -- (-1/9,1/18);
\draw (0,1/2) -- (-1/8,1/8);
\draw (0,1/2) -- (-1/4,0);
\draw[fill] (1/5,-1/10) circle (.1pt);
\draw[fill] (1/4,-1/4) circle (.1pt);
\draw[fill] (-1/3,-1/6) circle (.1pt);
\draw[fill] (-1/8,0) circle (.1pt);
\draw[->] (1/11,1/22 ) -- (1/11,-1/3);
\draw[->] (-1/11,1/22 ) -- (-1/11,-1/3);
\draw[->] (-1/10,0) -- (-1/10,-1/3);
\draw[->] (-1/9,-1/18 ) -- (-1/9,-1/3);
\draw[->] (-1/8,-1/8 ) -- (-1/8,-1/3);
\draw[->] (1/10,0) -- (1/10,-1/3);
\draw[->] (1/9,-1/18) -- (1/9,-1/3);
\fill[blue, nearly transparent] (0, 1/2) -- (1/11,1/22) -- (1/11, -1/3) -- (-1/11,-1/3) -- (-1/11,1/22) -- cycle;
\end{scope}

\begin{scope}[black] 
\draw (1/5,-1/10) node[right]{$\left(\frac{1}{5},\frac{-1}{10}\right)$};
\draw (1/4,-1/4) node[right]{$\left(\frac{1}{4},\frac{-1}{4}\right)$};
\draw (-1/3,-1/6) node[right]{$\left(\frac{-1}{3},\frac{-1}{6}\right)$};
\draw (-1/8,0) node[left]{ $\left(\frac{-1}{8},0\right)$};
\end{scope}

\end{tikzpicture}
\end{center}

We next treat the case $\lambda^\bullet=(\lambda^0, \lambda^1)$. We mark certain boxes as follows:
$$\lambda^0=
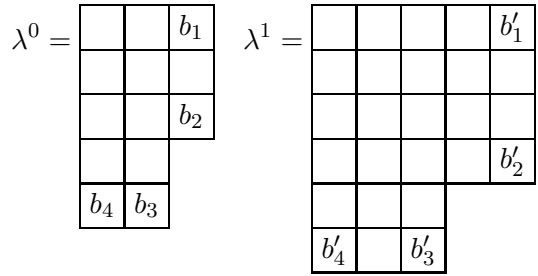
\begin{ytableau}
 \ & &  b_1\\
 & &   \\
 & & b_2  \\
 &  \\
 b_4 &  b_3 \\
\end{ytableau} \quad
\lambda^1=\begin{ytableau}
 \ & & & &  b_1' \\
 & & & &  \\
 & & & &  \\
 & & & &  b_2' \\
 & &  \\
 b_4' & &  b_3'
\end{ytableau}$$ Here $b_1$ is the box of $\lambda^0$ of largest content, $b_2$ is the removable box of largest content, $b_3$ is the removable box of smallest content, and $b_4$ is the box of smallest content. We similarly define the boxes $b_i'$ of $\lambda^1$. We do not state the classification of $c$ so that $L_c(\lambda^\bullet)$ is diagonalizable here; it is quite complicated. As it turns out, the classification of $c$ for which $L_c(\lambda^\bullet)$ is unitary does does not require it and is actually simpler to state. 

\begin{corollary} \label{B corollary 2}
With notation as above the unitary spectrum $U(\lambda^\bullet)$ is the union of the following sets:
\begin{enumerate}
\item[(a)] The parallelogram 
$$d+(\mathrm{ct}(b_1)-\mathrm{ct}(b_4')+1) c \leq \frac{1}{2} \quad d+(\mathrm{ct}(b_4)-\mathrm{ct}(b_1')-1) c \leq \frac{1}{2}$$
$$-d+(\mathrm{ct}(b_1')-\mathrm{ct}(b_4)+1) c \leq \frac{1}{2} \quad -d+(\mathrm{ct}(b_4')-\mathrm{ct}(b_1)-1) c \leq \frac{1}{2}$$
\item[(b)] For each integer $l$ with $\mathrm{ct}(b_2)-\mathrm{ct}(b_4')+1 \leq l \leq \mathrm{ct}(b_1)-\mathrm{ct}(b_4')+1$, the line segment $S_l$ consisting of points $(c,d)$ with
$$d+lc=\frac{1}{2}, \quad c \geq 0, \quad \text{and} \quad -d+(\mathrm{ct}(b_1')-\mathrm{ct}(b_4)+1)c \leq \frac{1}{2},$$
\item[(c)] For each integer $l$ with $\mathrm{ct}(b_2')-\mathrm{ct}(b_4)+1 \leq l \leq \mathrm{ct}(b_1')-\mathrm{ct}(b_4)+1$ a line segment $S_l'$ consisting of the points $(c,d)$ with
$$-d+lc=\frac{1}{2}, \quad c \geq 0, \quad \text{and} \quad d+(\mathrm{ct}(b_1)-\mathrm{ct}(b_4')+1) c \leq \frac{1}{2},$$
\item[(d)] For each integer $l$ with $\mathrm{ct}(b_4)-\mathrm{ct}(b_1')-1 \leq l \leq \mathrm{ct}(b_3)-\mathrm{ct}(b_1')-1$, a line segment $T_l$ consisting of points $(c,d)$ with
$$d+lc=\frac{1}{2}, \quad c \leq 0, \quad \text{and} \quad -d+(\mathrm{ct}(b_4')-\mathrm{ct}(b_1)-1)c \leq \frac{1}{2},$$
\item[(e)] For each integer $l$ with $\mathrm{ct}(b_4')-\mathrm{ct}(b_1)-1 \leq l \leq \mathrm{ct}(b_3')-\mathrm{ct}(b_1)-1$, a line segment $T_l'$ consisting of points $(c,d)$ with
$$-d+lc=\frac{1}{2}, \quad c \leq 0, \quad \text{and} \quad d+(\mathrm{ct}(b_4)-\mathrm{ct}(b_1')-1)c \leq \frac{1}{2},$$
\item[(f)] For each pair $(l,m)$ of integers with $\mathrm{ct}(b_2)-\mathrm{ct}(b_4')+1 \leq l \leq \mathrm{ct}(b_1)-\mathrm{ct}(b_4')+1$ and $\mathrm{ct}(b_2')-\mathrm{ct}(b_4)+1 \leq m \leq \mathrm{ct}(b_1')-\mathrm{ct}(b_4)+1$, the point $P_{l,m}$ solving the equations
$$d+lc=\frac{1}{2} \quad \text{and} \quad -d+mc=\frac{1}{2},$$
\item[(g)] For each pair $(l,m)$ of integers with $\mathrm{ct}(b_4)-\mathrm{ct}(b_1')-1 \leq l \leq \mathrm{ct}(b_3)-\mathrm{ct}(b_1')-1$ and $\mathrm{ct}(b_4')-\mathrm{ct}(b_1)-1 \leq m \leq \mathrm{ct}(b_3')-\mathrm{ct}(b_1)-1$, the point $Q_{l,m}$ solving the equations
$$d+lc=\frac{1}{2} \quad \text{and} \quad -d+mc=\frac{1}{2}.$$
\end{enumerate}
\end{corollary}

\begin{proof}
If $c=0$ then by Theorem \ref{unitary thm} $L_c(\lambda^\bullet)$ is unitary if and only if $-1/2 \leq d \leq 1/2$. So we may assume $c > 0$. Assume first that if $L_c(\lambda^\bullet)$ is unitary. If $$d+(\mathrm{ct}(b_1)-\mathrm{ct}(b_4')+1)c > 1/2$$ then there must exist a blocking sequence $B$ for $(b_1,b_4)$ with $c \in L_B$. This $B$ must be of the form $(b,b')$ with $b \geq b_1$ and $b' \leq b_4'$, and hence
$$d+\ell c=1/2, \quad \text{where} \quad \ell=\mathrm{ct}(b)-\mathrm{ct}(b')+1.$$ We observe that
$$\ell \leq \mathrm{ct}(b_1)-\mathrm{ct}(b_4')+1$$ since $\mathrm{ct}(b_1) \geq \mathrm{ct}(b)$ and $\mathrm{ct}(b_4') \leq \mathrm{ct}(b')$. Moreover we must have $\mathrm{ct}(b_2)-\mathrm{ct}(b_4')+1 \leq \ell$ since otherwise $$\mathrm{ct}(b)-\mathrm{ct}(b') < \mathrm{ct}(b_2)- \mathrm{ct}(b_4') \quad \implies \quad \mathrm{ct}(b)-\mathrm{ct}(b_2) < \mathrm{ct}(b') -\mathrm{ct}(b_4'),$$ so there is a box $b'' \in \lambda^1$ with $\ell=\mathrm{ct}(b_2)-\mathrm{ct}(b'')$.  But this implies $k_c(b_2)=1$ and hence $L_c(\lambda^\bullet)$ is not diagonalizable, contradiction. 

We have proved that if $$d+(\mathrm{ct}(b_1)-\mathrm{ct}(b_4')+1)c > 1/2$$ then there is $\ell$ with
$$\mathrm{ct}(b_2)-\mathrm{ct}(b_4')+1 \leq \ell \leq \mathrm{ct}(b_1)-\mathrm{ct}(b_4')+1$$ and
$$d+\ell c=1/2.$$ By symmetry, if $-d+(\mathrm{ct}(b_1')-\mathrm{ct}(b_4)+1) c > 1/2$ then there is some $\ell$ with 
$$\mathrm{ct}(b_2')-\mathrm{ct}(b_4)+1 \leq \ell \leq \mathrm{ct}(b_1')-\mathrm{ct}(b_4)+1$$ and
$$-d+\ell c=1/2.$$ This shows that, for $c>0$, if we are not in case (a) of the present corollary then we are in one of cases (b), (c), or (f) (replacing $c$ by $-c$ and transposing $\lambda^\bullet$ gives the others cases).

Conversely, assume we are in one of cases (a), (b), (c), or (f) (the others may be treated by replacing $c$ by $-c$ as above). We will show first that $L_c(\lambda^\bullet)$ is diagonalizable. In each case we have
$$d+(\mathrm{ct}(b_2)-\mathrm{ct}(b_4')+1)c \leq 1/2 \quad \text{and} \quad -d+(\mathrm{ct}(b_2')-\mathrm{ct}(b_4)+1)c \leq 1/2$$ which implies
$$(\mathrm{ct}(b_2)-\mathrm{ct}(b_4')+1+\mathrm{ct}(b_2')-\mathrm{ct}(b_4)+1)c \leq 1.$$ Thus
$$c \leq \frac{1}{\mathrm{ct}(b_2)-\mathrm{ct}(b_4)+1} \quad \text{and} \quad c \leq \frac{1}{\mathrm{ct}(b_2')-\mathrm{ct}(b_4')+1}.$$ Together these inequalities imply $k_c(b)=\infty$ for all removable boxes $b$ of $\lambda^\bullet$, and hence by Theorem \ref{diag thm} $L_c(\lambda^\bullet)$ is diagonalizable. 

In case (a) of the present corollary, the conditions (b) and (c) of Theorem \ref{diag thm} automatically hold. Suppose we are in case (b). We have $d+\ell c=1/2$ , for some $\mathrm{ct}(b_2)-\mathrm{ct}(b_4')+1 \leq \ell \leq \mathrm{ct}(b_1)-\mathrm{ct}(b_4')+1$, $c > 0$, and $-d+(\mathrm{ct}(b_1')-\mathrm{ct}(b_4)+1)c \leq 1/2$. Let $b$ be the box with $b_1 \leq b \leq b_2$ and $\mathrm{ct}(b)-\mathrm{ct}(b_4')+1=\ell$. Thus $B=(b,1)$ is a blocking sequence with $c \in L_B$ for all $(b',j)$ and all $(b',b'')$ with $b' \leq b$. Suppose now that $b',b'' \in \lambda^0$ with $$(\mathrm{ct}(b')-\mathrm{ct}(b'')+1)c >1.$$ Since
$$(\mathrm{ct}(b)-\mathrm{ct}(b_4')+1+\mathrm{ct}(b_1')-\mathrm{ct}(b_4)+1)c \leq 1,$$ if $\mathrm{ct}(b') \leq \mathrm{ct}(b)$ we obtain
$$1<(\mathrm{ct}(b')-\mathrm{ct}(b'')+1)c \leq (\mathrm{ct}(b)-\mathrm{ct}(b_4')+1+\mathrm{ct}(b_1')-\mathrm{ct}(b_4)+1)c \leq 1.$$ Thus $\mathrm{ct}(b') > \mathrm{ct}(b)$ and hence $b' \leq b$, so $B=(b,1)$ is a blocking sequence for $(b',b'')$ with $c \in L_B$. Likewise if $b',b'' \in \lambda^1$ then since $\mathrm{ct}(b') \leq \mathrm{ct}(b_1')$
$$(\mathrm{ct}(b')-\mathrm{ct}(b'')+1)c \leq (\mathrm{ct}(b)-\mathrm{ct}(b_4')+1+\mathrm{ct}(b_1')-\mathrm{ct}(b_4)+1)c \leq 1.$$ Suppose now that $b' \in \lambda^1$ and $b'' \in \lambda^0$. Then
$$-d+(\mathrm{ct}(b')-\mathrm{ct}(b'')+1)c \leq -d+(\mathrm{ct}(b_1')-\mathrm{ct}(b_4)+1)c \leq 1/2$$ and
$$-d+\mathrm{ct}(b') c \leq -d+(\mathrm{ct}(b_1')-\mathrm{ct}(b_4)+1)c \leq 1/2.$$ Finally, for $b' \in \lambda^0$ and $b'' \in \lambda^1$ with $\mathrm{ct}(b') > \mathrm{ct}(b)$ we have $b' \leq b$ and hence $B=(b,1)$ is a blocking sequence for $(b',1)$ and $(b',b'')$ with $c \in L_B$, while if $\mathrm{ct}(b') \leq \mathrm{ct}(b)$ then
$$d+(\mathrm{ct}(b')-\mathrm{ct}(b'')+1)c \leq d+(\mathrm{ct}(b)-\mathrm{ct}(b_4')+1)c =1/2$$ and
$$d+\mathrm{ct}(b') c \leq d+(\mathrm{ct}(b)-\mathrm{ct}(b_4')+1)c =1/2.$$ It follows that conditions (b) and (c) from Theorem \ref{unitary thm} hold.

Case (c) of the present corollary implies unitarity by symmetry (interchanging $d$ and $-d$, and $\lambda^0$ and $\lambda^1$), and case (f) is completed by analogous arguments.
\end{proof}

A picture of the unitary spectrum for $\lambda^\bullet=((3,3,3,2),(5,5,5,5,3,3))$ follows. The area shaded light blue is the region in which the standard module itself is unitary.

\newpage

\begin{center}
\begin{tikzpicture}[scale=17]
\tikzstyle{axes}=[]
\tikzstyle{wall}=[thick]
\tikzstyle{relevant wall}=[very thick]
\tikzstyle{dot}=[fill]
\begin{scope}[style=axes]
\draw[->] (-1/3,0) -- (1/3,0) node[right] {$c$} coordinate(c1 axis);
\draw[->] (0,-3/5) -- (0,3/5) node[above] {$d$} coordinate(c2 axis);
\end{scope}

\begin{scope}[very thick,blue,auto=left]
\fill[blue, nearly transparent] (0, 1/2) -- (1/17,1/34) -- (0,-1/2) -- (-1/17,-1/34) -- cycle;
\draw (0, 1/2) -- (1/17,1/34);
\draw (0, -1/2) -- (1/17,1/34);
\draw (0, 1/2) -- (-1/17,-1/34);
\draw (0, -1/2) -- (-1/17,-1/34);
\draw (0,1/2) -- (1/16,1/16);
\draw (0,1/2) -- (1/15,1/10);
\draw (0,-1/2) -- (1/16,0);
\draw (0,-1/2) -- (1/15,-1/30);
\draw (0,-1/2) -- (1/14,-1/14);
\draw (0,1/2) -- (-1/16,0);
\draw (0,-1/2) -- (-1/16,-1/16);
\draw (0,-1/2) -- (-1/15,-1/10);
\draw[fill] (1/15,1/30) circle (.05 pt);
\draw[fill] (1/14,0) circle (.05 pt);
\draw[fill] (1/14,1/14) circle (.05 pt);
\draw[fill] (1/13,1/26) circle (.05 pt);
\draw[fill] (1/13,-1/26) circle (.05 pt);
\draw[fill] (1/12,0) circle (.05 pt);
\draw[fill] (-1/15,-1/30) circle (.05 pt);
\draw[fill] (-1/14,-1/14) circle (.05 pt);
\end{scope}
\end{tikzpicture}
\end{center}

\def\cprime{$'$} \def\cprime{$'$}

\end{document}